\documentclass[a4paper,12pt,intlimits,oneside]{amsart}
\usepackage{amsmath}
\usepackage{amsthm}
\usepackage{latexsym}
\usepackage{amssymb}
\usepackage{xcolor}
\usepackage{graphicx}
\usepackage[colorlinks=true]{hyperref}
\numberwithin{figure}{section}
\def\R{{\mathbb R}}
\def\C{{\mathbb C}}

\def\T{{\mathbb T}}
\def\Z{{\mathbb Z}}

\def\N{{\mathbb N}}

\def\e{\varepsilon}

\def\build#1_#2^#3{\mathrel{
\mathop{\kern 0pt#1}\limits_{#2}^{#3}}}
\def\td_#1,#2{\mathrel{\mathop{\build\longrightarrow_{#1\rightarrow #2}^{}}}}
\DeclareFontFamily{U}{MnSymbolC}{}
\DeclareSymbolFont{MnSyC}{U}{MnSymbolC}{m}{n}
\DeclareFontShape{U}{MnSymbolC}{m}{n}{
    <-6>  MnSymbolC5
   <6-7>  MnSymbolC6
   <7-8>  MnSymbolC7
   <8-9>  MnSymbolC8
   <9-10> MnSymbolC9
  <10-12> MnSymbolC10
  <12->   MnSymbolC12}{}
\DeclareMathSymbol{\intprod}{\mathbin}{MnSyC}{'270}
\newtheorem{theorem}{Theorem}
\newtheorem{corollary}{Corollary}
\newtheorem{proposition}{Proposition}
\newtheorem{lemma}{Lemma}
\newtheorem{remark}{Remark}
\newtheorem{definition}{Definition}
\begin{document}
\title[Flow map of the BO equation]{On the flow map of the Benjamin-Ono equation on the torus}
\author[P. G\'erard]{Patrick G\'erard}
\address{Laboratoire de Math\'ematiques d'Orsay, Univ. Paris-Sud, CNRS, Universit\'e Paris--Saclay, 91405 Orsay, France} \email{{\tt patrick.gerard@math.u-psud.fr}}
\author[T. Kappeler]{Thomas Kappeler}
\address{Institut f\"ur Mathematik, Universit\"at Z\"urich, Winterthurerstrasse 190, 8057 Zurich, Switzerland} 
\email{{\tt thomas.kappeler@math.uzh.ch}}
\author[P. Topalov]{Peter Topalov}
\address{Department of Mathematics, Northeastern University,
567 LA (Lake Hall), Boston, MA 0215, USA}
\email{{\tt p.topalov@northeatsern.edu}}

\subjclass[2010]{ 37K15 primary, 47B35 secondary}

\date{December 5,  2019}

\begin{abstract}
We prove that for any $0 < s < 1/2$,  
the Benjamin--Ono equation on the torus
is globally in time $C^0-$well-posed on 
the Sobolev space $H^{-s}(\T, \R)$, in the sense that the solution map, 
which is known to be defined for smooth data, continuously extends 
to $H^{-s}(\T, \R)$. The solution map does not extend continuously
to $H^{-s}(\T, \R)$ with $s > 1/2$. 
Hence the critical Sobolev exponent $s_c=-1/2$
of the Benjamin--Ono equation is the threshold 
for well-posedness on the torus. 
The obtained solutions are almost periodic in time. 
Furthermore, we prove that  the traveling wave solutions of the Benjamin--Ono equation 
on the torus are orbitally stable in $H^{-s}(\T ,\R)$ for any $0 \le s < 1/2$.
\end{abstract}

\keywords{Benjamin--Ono equation, well-posedness,
critical Sobolev exponent, almost periodicity of solutions, 
orbital stability of traveling waves}

\thanks{We would like to warmly thank J.C. Saut for
very valuable discussions and for
making us aware of many references, in particular
\cite{AN}. We also thank T. Oh for bringing reference 
\cite{AH} to our attention. 
T.K. partially supported by the Swiss National Science Foundation.
P.T. partially supported by the Simons Foundation, Award \#526907.}

\maketitle

\tableofcontents

\medskip

\section{Introduction}\label{Introduction}

In this paper we consider the Benjamin-Ono (BO) equation on the torus,
\begin{equation}\label{BO}
\partial_t v = H\partial^2_x v - \partial_x (v^2)\,, \qquad x \in \T:= \R/2\pi\Z\,, \,\,\, t \in \R,
\end{equation}
where $v\equiv v(t, x)$ is real valued and $H$ denotes the Hilbert transform, defined for $f = \sum_{n \in \mathbb Z} \widehat f(n) e^{ i n x} $,
$\widehat f(n) = \frac{1}{2\pi}\int_0^{2\pi} f(x) e^{-  i n x} dx$,  by
$$
H f(x) := \sum_{n \in \Z} -i \ \text{sign}(n) \widehat f(n) \ e^{inx}
$$
with $\text{sign}(\pm n):= \pm 1$ for any $n \ge 1$, whereas $\text{sign}(0) := 0$. This pseudo-differential equation ($\Psi$DE)
in one space dimension has been introduced by Benjamin \cite{Benj} and Ono \cite{Ono} to model long, uni-directional internal gravity waves 
in a two-layer fluid. It has been extensively studied, 
both on the real line $\R$ and on the torus $\T$. 
For an excellent survey, including the derivation of  \eqref{BO}, we refer to the recent article by Saut \cite{Sa}.

Our aim is to study  low regularity solutions of the $BO$ equation on $\T$. To state our results, we first need to review
some classical results on the well-posedness problem of \eqref{BO}.
Based on work of Saut \cite{Sa0}, 
Abelouhab, Bona, Felland, and Saut
proved in \cite{ABFS} that
for any $s \ge 3/2$, equation \eqref{BO} is globally in time 
well-posed on the Sobolev space $H^s_r \equiv H^s(\T, \R) $ 
(endowed with the standard norm $\| \cdot \|_s$, defined by \eqref{Hs norm} below), meaning the following:
\begin{itemize}
\item[(S1)]
{\em Existence and uniqueness of classical solutions:} For any initial data $v_0 \in H^{s}_r$, there exists 
a unique curve $v : \R \to H^s_r$ in
$C(\R, H^s_r) \cap C^1(\R, H^{s-2}_r)$ so that
$v(0) = v_0$ and for any $t \in \R$, equation \eqref{BO} is satisfied in $H^{s-2}_r$.   
(Since $H^s_r$ is an algebra, one has 
$\partial_xv(t)^2 \in H^{s-1}$ for any time $t \in \R$.)
\item[(S2)]{\em Continuity of solution map:}
The solution map 
$\mathcal S : H^s_r \to C(\mathbb R, H^s_r)$
is continuous, meaning that for any $v_0 \in H^s_r,$
$T > 0$, and $\varepsilon > 0$ there exists $\delta > 0,$ 
so that for any $w_0 \in H^s_r$ with $\|w_0 - v_0 \|_s < \delta$,
the solutions $w(t) = \mathcal S(t, w_0)$ and 
$v(t)= \mathcal S(t, v_0)$ of \eqref{BO}
with initial data $w(0) = w_0$ and, respectively, $v(0) = v_0$
satisfy $\sup_{|t| \le T} \| w(t) - v(t) \|_s \le  \varepsilon$.
\end{itemize}
In a straightforward way one verifies that
\begin{equation}\label{CL}
 \mathcal H^{(-1)}(v) := \langle v | 1 \rangle  \, , \qquad
 \mathcal H^{(0)}(v) :=  \frac{1}{2} \langle v | v \rangle  
\end{equation}
are integrals of the above solutions of \eqref{BO},
referred to as classical solutions, 
where $\langle \cdot  \, |  \, \cdot \rangle $ denotes the $L^2-$inner product, 
\begin{equation}\label{L2 inner product}
\langle f | g \rangle =  \frac{1}{2\pi} \int_0^{2\pi} f \overline g dx \, .
\end{equation}
In particular it follows that for any $c \in \R$ 
and any $s \ge 3/2$,
the affine space $H^s_{r,c}$ is left invariant by $\mathcal S$
where for any $\sigma \in \R$
\begin{equation}\label{affine spaces}
H^\sigma_{r,c} :=  
\{ w \in H^\sigma_r \, : \, \langle w | 1 \rangle =c \}\, .
\end{equation}

\noindent
In the sequel, further progress has been made 
on the well-posedness of \eqref{BO} 
on Sobolev spaces of low regularity. 
The best results so far in this direction
were obtained by Molinet  by using the gauge transformation introduced by Tao \cite{Tao}. 
Molinet's results in \cite{Mol} ( cf. also \cite{MP}) imply that the solution 
map $\mathcal S,$ introduced in $(S2)$ above, 
continuously extends to any
Sobolev space $H^s_r$ with $0 \le s \le 3/2$. 
More precisely, for any such $s$,
$\mathcal S: H^s_r \to C(\R, H^s_r)$ is continuous
and for any $v_0 \in H^s_r$, $ \mathcal S(t, v_0)$ satisfies equation \eqref{BO}  
in $H^{s-2}_r$. The fact that $\mathcal S$ continuously
extends to $L^2_{r} \equiv H^0_{r}$, 
$\mathcal S : L^2_{r} \to C(\R, L^2_{r})$,
 can also be deduced by methods
recently developed in \cite{GK}. Furthermore, one infers from \cite{GK}
that any solution  $ \mathcal S(t, v_0)$ with initial data $v_0 \in L^2_r$
can be approximated in $C(\R, L^2_r)$ by  solutions of \eqref{BO} which are rational functions of $\cos x, \sin x$.
In this paper we will refer to these solutions as rational solutions.

In this paper we show that the BO equation is well-posed in the Sobolev space
$H^{-s}_r$ for any $0 < s < 1/2$ and that this result is sharp. i.e.,
that the critical Sobolev exponent $s_c=-1/2$ is the threshold for well-posedness. 
Since the nonlinear term $\partial_x v^2$ in equation \eqref{BO} is not 
well-defined for elements in $H^{-s}_r$, we first need to define what
we mean by a solution of \eqref{BO} in such a space.

\begin{definition}\label{def solution}
Let $s \ge 0$.
A continuous curve $\gamma: \R \to H^{-s}_r$ 
with $\gamma(0) = v_0$ for a given $v_0 \in H^{-s}_r$,
is called a global in time solution of the BO equation in 
$H^{-s}_r$ with initial data $v_0$ if for any
sequence $(v_0^{(k)})_{k \ge 1}$ in $H^\sigma_r$ with $\sigma > 3/2,$
which converges to $v_0$ in $H^{-s}_r$, the corresponding
sequence of classical solutions $\mathcal S(\cdot, v_0^{(k)})$
converges to $\gamma$ in $C(\R, H^{-s}_{r})$.
The solution $\gamma$ is denoted by 
$\mathcal S(\cdot, v_0)$.
\end{definition}
\noindent
We remark that for any $v_0 \in L^2_r$, the solution
$\mathcal S(\cdot, v_0)$ in the sense of Definition \ref{def solution}
coincides with the solution obtained by Molinet in \cite{Mol}.
\begin{definition}\label{def well-posedness}
Let $s \ge 0$.
Equation \eqref{BO} is said to be globally $C^0-$well-posed
in $H^{-s}_r$ if the following holds:
\begin{itemize}
\item[(i)]
For any $v_0 \in H^{-s}_r$, there exists
a global in time solution of \eqref{BO} 
with initial data $v_0$
in the sense of Definition \ref{def solution}.
\item[(ii)]
The solution map $\mathcal S : H^{-s}_r \to C(\R, H^{-s}_r)$
is continuous, i.e. satisfies (S2).
\end{itemize}
\end{definition}
 Our main results are the following ones:
\begin{theorem}\label{Theorem 1}
For any $0 \le s < 1/2$,  the Benjamin-Ono equation is globally $C^0-$well-posed
on $H^{-s}_r$ in the sense of Definition \ref{def well-posedness}.  For any $c \in \R$,  $t\in \R $,  the flow map $\mathcal S^t=S(t, \cdot)$  leaves
the affine space $H^{-s}_{r,c}$,  introduced in \eqref{affine spaces}, invariant.\\
Furthermore, there exists an integral 
$I_s : H^{-s}_r \to \R_{\ge 0}$ of \eqref{BO} 
satisfying
$$
\| v \|_{-s} \le I_s(v)\, , \quad \forall v \in H^{-s}_r \, .
$$
In particular, one has
$$
\sup_{t \in \R} \|\mathcal S(t, v_0)\|_{-s}
\le I_s(v_0)\ , \quad \forall v_0 \in H^{-s}_r\, .
$$
\end{theorem}
\begin{remark}
$(i)$ Note that global $C^0$--wellposedness implies the group property $S^{t_1}\circ S^{t_2}=S^{t_1+t_2}$. Consequently, $S^t$ is a homeomorphism of $H^{-s}$. \\
$(ii)$ Since by \eqref{CL}, the $L^2-$norm is an integral of \eqref{BO}, $I_s$ in the case $s=0$
can be chosen as $I_0(v):= \| v \|_0^2$. The
definition of $I_s$ for $0 < s < 1/2$ can be found in 
Remark \ref{def integral I_s} in 
Section \ref{Birkhoff map}.\\
$(iii)$ By the Rellich compactness theorem, $S^t $ is also  weakly sequentially continuous, in particular on $L^2_{r,0}$. Note that this  contradicts the result of \cite{Mol1}. Very recently, however, 
an error in the proof of this statement has been found,
leading to the withdrawal of the paper (cf. arXiv:0811.0505). 
A proof of this weak continuity property was indeed the starting point 
of the present paper.
\end{remark}
\begin{remark}\label{illposedness}
 It was already observed in \cite{AH} that
the solution map $\mathcal S$ does not continuously extend
to $H^{-s}_{r}$ with $s > 1/2$. More precisely, for any $c \in \R$,
there exists a sequence
$(v_0^{(k)})_{k \ge 1}$ in $\cap_{n \ge 0} H^n_{r,c}$,
which for any $s > 1/2$
converges to an element $v_0$ in $H^{-s}_{r,c}$
so that for any $t \ne 0$, the
sequence $\mathcal S(t, v_0^{(k)})$ does not 
even converge to a distribution on $\T$
in the sense of distributions. Since the methods developed in this paper allow us to give a short proof of this result,
we include it in Section \ref{Proofs of main results}. 
\end{remark}
 One of the key ingredients of our proof of Theorem \ref{Theorem 1} 
 are explicit formulas for the frequencies of the Benjamin-Ono equation, defined by
 \eqref{frequencies in Birkhoff 0} below. 
 They are not only used to prove the global wellposedness results for \eqref{BO},
 but at the same time allow to obtain the following qualitative properties of solutions of \eqref{BO}.
  
\begin{theorem}\label{Theorem 3}
For any $v_0 \in H^{-s}_{r, c}$ with $0 < s < 1/2$ and $c \in \R$,
the solution $t \mapsto \mathcal S(t, v_0)$
has the following properties:\\
$(i)$ The orbit $\{ S(t, v_0) \, : \, t \in \R \}$
is relatively compact in $H^{-s}_{r, c}$.\\
$(ii)$ The solution $t \mapsto \mathcal S(t, v_0)$
is almost periodic in $H^{-s}_{r, c}$.
\end{theorem}
\begin{remark}
For $s=0,$ results corresponding to the ones of Theorem \ref{Theorem 3} have been obtained in \cite{GK}.
\end{remark}

In \cite{AT}, Amick\&Toland characterized the travelling wave solutions
of \eqref{BO}, originally found by Benjamin \cite{Benj} . It was shown in \cite[Appendix B]{GK} that they coincide
with the so called one gap solutions, described explicitly in \cite{GK}.
Note that one gap potentials are rational solutions of \eqref{BO} and
evolve in $H^s_r$ for any $s \ge 0$.
In \cite[Section 5.1]{AN} Angulo Pava\&Natali proved that every travelling
wave solution of \eqref{BO} is orbitally stable in $H^{1/2}_r$. Our newly developed methods
allow to complement their result as follows:
\begin{theorem}\label{Theorem 4}
Every travelling wave solution of the BO equation
is orbitally stable in $H^{-s}_r $ for any $0 \le s < 1/2$. 
\end{theorem}

\medskip
\noindent
Let us comment on the novelty of our results.
\begin{itemize}
\item[1.]
A straightforward computation shows that $s_c = - 1/2$ is
the critical Sobolev exponent of the Benjamin-Ono equation.
Hence Theorem \ref{Theorem 1} and Remark 1(iv) 
imply that the threshold of well-posedness 
of \eqref{BO} on the scale of Sobolev spaces $H^s_r$ is 
given by the critical Sobolev exponent $s_c$.
\item[2.] In a recent, very interesting paper \cite{Tal}, 
Talbut proved by the method of perturbation determinants,
developed for the KdV and the NLS equations by Killip, Visan, and Zhang in \cite{KVZ},
that for any $0 < s < 1/2,$ there exists a constant 
$C_s>0$, only depending on $s$,
so that any sufficiently smooth solution $t \to v(t)$ 
of \eqref{BO} satisfies the estimate
$$
\sup_{t\in \R}\| v(t)\|_{-s}\leq  C_s 
\big( 1 + \|v(0)\|_{-s}^{\frac{2}{1-2s}} \big)^s
\|v(0)\|_{-s} \, .
$$
Note that our method allows to prove that the solution
map $\mathcal S$ continuously extends to $H^{-s}_r$.
Actually, it allows to achieve much more 
by constructing a nonlinear Fourier transform,
also referred to as Birkhoff map
(cf. Section \ref{Birkhoff map}), 
which is also of great use for proving 
Theorem \ref{Theorem 3} and Theorem \ref{Theorem 4}.
The integral $I_s$ of Theorem \ref{Theorem 1}$(iv)$
is taylored to show that for any $0 < s < 1/2$,
the Birkhoff map
$\Phi: H^{-s}_{r,0} \to h^{1/2 -s}_+$ 
is onto  
(cf. Theorem \ref{extension Phi}).\\
For recent work
on a priori bounds for Sobolev norms of smooth solutions
of the KdV equation and/or the NLS equation in 1d
and their applications to the initial value problem
of these equations see also \cite{KV}, \cite{KT}.
\item [3.] Using a probabilistic approach developed by Tzvetkov and Visciglia \cite{TV}, Y. Deng \cite{D} 
proved wellposedness result for the BO equation on the torus for almost every data with respect to a measure which is supported by $H^{-\e}_r$ for any $\e >0$ 
and has the property that $L^2_r$ has measure $0$. Our result provides a deterministic framework for these solutions.
\item[4.] A first version of this paper
appeared on arXiv in September 2019 -- see \cite{GKT}.
\end{itemize}

\smallskip

\noindent
{\em Outline of proofs.} The key idea is to 
construct for any $0 \le s < 1/2$,
globally defined canonical coordinates 
on $H^{-s}_{r,0}$ with the property that when
expressed in these coordinates, equation \eqref{BO}
can be solved by quadrature. 
Such coordinates are referred to as nonlinear Fourier coefficients or Birkhoff coordinates
and the corresponding transformation, denoted by $\Phi$, as Birkhoff map. 
Such a map was constructed on $L^2_{r,0} \equiv H^0_{r,0}$ in our previous work \cite{GK}.
In this paper we show that it can be continuously extended to the Sobolev spaces  $H^{-s}_{r,0}$
for any $0 < s < 1/2$. For this purpose we develop a new approach for
studying the Lax operator appearing in the Lax pair formulation
of \eqref{BO}.

In Section 2, we state our results on the extension of  Birkhoff map  $\Phi$ (cf. Theorem \ref{extension Phi})
 and discuss first applications. All these results are proved 
in Section \ref{extension of Phi, part 1}  and Section \ref{extension of Phi, part 2}.
In Section \ref{mathcal S_B},
we study the solution map $\mathcal S_B$
corresponding to the system of equations,
obtained when expressing \eqref{BO} in Birkhoff coordinates. 
The main point is to show that 
the frequencies of the BO equation, which have
been computed in our previous work \cite{GK}
on $L^2_{r,0}$, continuously extend to  $H^{-s}_{r, 0}$
for any $0 < s < 1/2$.  These results are then used
to study the solution map $\mathcal S$ of \eqref{BO}.
In the same section
we also introduce the solution map $\mathcal S_c$, 
related to the
equation \eqref{BO} when considered in the affine
space $H^s_{r,c}$ and study the
solution map $\mathcal S_{c, B}$, obtained by
expressing $\mathcal S_c$ in Birkhoff coordinates.
With all these preparations done,
we then prove our main results, stated in Theorem \ref{Theorem 1} -- Theorem \ref{Theorem 4},
 in Section \ref{Proofs of main results}.
We remark that the proof of Theorem \ref{Theorem 3} 
uses the same arguments as the one of Theorem 2 in  \cite{GK}, stating corresponding results for 
solutions of \eqref{BO} with initial data in $H^{0}_{r,0}$. In order to be comprehensive
and since the proof is short, we included it.

\smallskip

\noindent
{\em Related work.} 
By similar methods, results on global wellposedness of the type
stated in Theorem \ref{Theorem 1} have been obtained for other integrable PDEs such as the KdV, the KdV2,
the mKdV, and the defocusing NLS equations. In addition, a detailed analysis of the frequencies of these equations
allowed to prove in addition to the wellposedness results {\em qualitative properties} of solutions of these equations,
among them properties corresponding to the ones stated in Theorem \ref{Theorem 3} 
 -- see e.g. \cite{KT1},\cite{KT2}, \cite{KM1}, \cite{KM2}. 
 Very recently, new global wellposedness results for the KdV equation on the line were
obtained by Killip, Visan, and Zhang in \cite{KV}  by using  the integrable structure of the KdV equation 
in a novel way. By the same method, the authors also prove global wellposedness results for the
KdV equation on the torus (of the type stated in Theorem \ref{Theorem 1}) and for the NLS equation.
 
 Let us comment on the principal differences between the KdV equation and the Benjamin--Ono equation, 
 when viewed as integrable systems with a Lax pair formulation. One of the main differences is that the Lax operator $L$
 associated with the Benjamin--Ono equation is nonlocal. As a result, the spectral analysis of $L$ is of a quite different nature
 than the one of the Lax operator of the KdV equation, given by the Hill operator and hence
 being a differential operator of order two. One of the consequences of $L$ being nonlocal
 is that the study of the regularity of the Birkhoff map and of its restrictions to the scale of Sobolev spaces $H^s_{r,0}$, $s \ge 0,$
 is much more involved than in the case of the Birkhoff map of the KdV equation. We plan to address this issue in future work.
A second principal difference is that the BO frequencies are affine functionals of the symplectic actions whereas the KdV frequencies
are transcendental functionals of such actions, making it much more difficult to extend them to Sobolev spaces of functions of low regularity
or Sobolev spaces of distributions.
 A third major difference concerns finite gap potentials:  in the case of the BO equation, finite gap potentials are finite sums of Poisson kernels, 
 whereas in the case of the KdV equation, such potentials are given in terms of theta functions. 

\smallskip

\noindent
{\em Notation.}
By and large, we will use the notation established in
\cite{GK}. In particular, 
the $H^s-$norm of an element $v$ in the 
Sobolev space $H^s \equiv H^s(\T, \C)$, $s \in \R$,
will be denoted by $\|v\|_s$. It is defined by 
\begin{equation}\label{Hs norm}
\|v\|_s = 
\big( \sum_{n \in \Z} \langle n \rangle^{2s}
|\widehat v(n)|^2 \big)^{1/2}\, , \quad
 \langle n \rangle = \max\{1, |n|\}\, .
\end{equation}
For $\|v\|_0$, we usually write $\|v\|$. 
By $\langle \cdot | \cdot \rangle$, we will also
denote the extension of the $L^2-$inner product,
introduced in \eqref{L2 inner product}, 
to $ H^{-s}\times H^s$, $s \in \R$, by duality.
By $H_+$ we denote the Hardy space, consisting
of elements $f \in L^2(\T, \C) \equiv H^0$ with
the property that 
$ \widehat f(n) = 0$ for any $n < 0$.
More generally, for any $s \in \R$,  $H^{s}_+$
denotes the subspace of $H^s,$
consisting of elements $f \in H^s$ with the property
that $ \widehat f(n) = 0$ for any $n < 0$.



\section{The Birkhoff map}\label{Birkhoff map}

In this section we present our results on Birkhoff coordinates
which will be a key ingredient of the proofs of
Theorem \ref{Theorem 1} -- Theorem \ref{Theorem 4}. 
We begin by reviewing the results
on Birkhoff coordinates proved in \cite{GK}.
Recall that on appropriate Sobolev spaces,  
\eqref{BO} can be written in Hamiltonian form  
$$
\partial_t u = \partial_x (\nabla \mathcal H (u))\,, \qquad  \mathcal H (u):=  \frac{1}{2\pi}\int_0^{2 \pi} 
\big( \frac{1}{2} 
(|\partial_x|^{1/2} u)^2 - \frac{1}{3} u^3 
\big) dx
$$
where $|\partial_x|^{1/2}$ is the square root of the Fourier multiplier operator $|\partial_x|$ 
given by
$$
|\partial_x| f(x) = \sum_{n \in \Z} |n| \widehat f(n) e^{inx}\,.
$$
Note that the $L^2-$gradient $\nabla \mathcal H$ of $\mathcal H$ can be computed to be $|\partial_x| u - u^2$ and that $\partial_x \nabla \mathcal H$ is the
Hamiltonian vector field
corresponding to the Gardner bracket, defined for any two functionals $F, G : H^0_{r} \to \R$ with sufficiently regular $L^2-$gradients by
$$
 \{F, G \} := \frac{1}{2 \pi} \int_0^{2\pi} (\partial_x \nabla F) \nabla G dx\ .
$$
In \cite{GK}, it is shown that \eqref{BO} admits global Birkhoff coordinates and hence is an integrable $\Psi$DE in the strongest possible sense. To state this result in more detail, we first introduce some notation.
For any subset $J \subset \N_0 :=\mathbb Z_{\ge 0}$ 
and any $s \in \mathbb R$, 
$h^s(J) \equiv h^s (J, \mathbb C)$ denotes the weighted $\ell^2-$sequence space
$$
h^s(J) = \{ (z_n)_{n \in J} \subset \mathbb C \, : \, \| (z_n)_{ n \in J} \|_s < \infty  \}
$$
where
$$
\| (z_n)_{ n \in J} \|_s : = \big( \sum_{n \in J} \langle n \rangle^{2s} |z_n|^2 \big)^{1/2} \ , \quad  
\langle n \rangle := \text{ max} \{ 1, |n| \} \, .
$$
By $h^s(J, \R)$, we denote the real subspace of 
$h^s(J, \C)$, consisting of real sequences $(z_n)_{n \in J}$.
In case where $J = \mathbb N := \{ n \in \mathbb Z \, : \, n \ge 1 \}$ we write
$h^s_+$ instead of $h^s(\mathbb N)$.
If $s=0,$ we also write $\ell^2$ instead of $h^0$ and
$\ell^2_+$ instead of $h^0_+$.
In the sequel, we view $h^s_+$ as the $\R-$Hilbert space $h^s(\N, \R) \oplus h^s(\N, \R)$
by identifying a sequence $(z_n)_{n \in \N} \in h^s_+$ with the pair of sequences
$\big( ({\rm Re} \, z_n)_{n \in \N}, ({\rm Im} \, z_n)_{n \in \N} \big)$ in $h^s(\N, \R) \oplus h^s(\N, \R)$.
We recall that  $L^2_r = H^{0}_r$ and $L^2_{r,0} = H^0_{r,0}$.
The following result was proved in \cite{GK}:
\begin{theorem}\label{main result}(\cite[Theorem 1]{GK}) 
There exists a homeomorphism 
$$
\Phi : L^2_{r,0}\to h^{1/2}_+ \,, \, u \mapsto (\zeta_n(u))_{n \ge 1}
$$
so that the following holds:\\
(B1) For any $n \ge 1$, $\zeta_n:L^2_{r,0}\to \C $ is real analytic.\\
(B2) The Poisson brackets between the coordinate functions $\zeta_n$ are well-defined and for any $n, k \ge 1,$
\begin{equation}\label{standard bracket}
\{\zeta_n , \overline{\zeta_k} \} = - i \delta_{nk}\,, \qquad  \{\zeta_n , \zeta_k \} = 0\,.
\end{equation} 
It implies that the functionals $|\zeta_n|^2$, $n \ge 1$,
pairwise Poisson commute,
$$
\{|\zeta_n|^2 , |\zeta_k|^2 \} = 0\, , 
\quad \forall n,k \ge 1 \, .
$$
(B3) On its domain of definition, $\mathcal H \circ \Phi^{-1}$ is a (real analytic) function, which only depends on the actions $|\zeta_n|^2,$ $n \ge 1$. 
As a consequence, for any $n \ge 1$,
$|\zeta_n|^2$ is an integral of $\mathcal H  \circ \Phi^{-1}$,
$\{ \mathcal H \circ \Phi^{-1}, |\zeta_n|^2 \} = 0$.
\\
The coordinates $\zeta_n$, $n \ge 1$, are referred to as complex Birkhoff coordinates and the functionals 
$|\zeta_n|^2$, $n \ge 1$,  as action variables.
\end{theorem}
\begin{remark}\label{RemarkThm1}
$(i)$ When restricted to submanifolds of finite gap potentials (cf. \cite[Definition 2.2 ]{GK}),
the map $\Phi$ is a canonical, real analytic diffeomorphism
onto corresponding Euclidean spaces -- see 
\cite[Theorem 3]{GK} for details.\\
$(ii)$  For any bounded subset $B$ of $L^2_{r,0}$, 
the image  $\Phi(B)$ by $\Phi$ is bounded in $h^{1/2}_+$.
This is a direct consequence of the trace formula, saying that for any $u \in L^2_{r,0}$ (cf. \cite[Proposition 3.1]{GK}), 
\begin{equation}\label{trace formula}
\|u \|^2 = 2 \sum_{n =1}^\infty n |\zeta_n|^2 \ .
\end{equation}  
\end{remark}

Theorem \ref{main result} together with Remark \ref{RemarkThm1}(i) can be used to solve 
the initial value problem of
\eqref{BO} in $L^2_{r,0}$.
Indeed, by approximating a given initial data in $L^2_{r,0}$ by finite gap potentials (cf. \cite[Definition 2.2 ]{GK}),
one concludes from \cite[Theorem 3]{GK} and Theorem \ref{main result} that equation $\eqref{BO}$, 
when expressed in the Birkhoff coordinates $\zeta = (\zeta_n)_{n \ge 1}$, reads
\begin{equation}\label{BO in Birkhoff}
\partial_t \zeta_n = 
\{\mathcal H \circ \Phi^{-1}, \zeta_n  \} =
i \omega_n \zeta_n \, , 
\quad \forall n \ge 1\, ,
\end{equation}
where 
$\omega_n$, $n \ge 1$, are the BO frequencies,
\begin{equation}\label{frequencies in Birkhoff 0}
\omega_n = 
\partial_{|\zeta_n|^2} \mathcal H \circ \Phi^{-1} \, .
\end{equation}
Since the frequencies only depend on the actions $|\zeta_k|^2$,
$k \ge 1$, they are conserved and hence \eqref{BO in Birkhoff}
can be solved by quadrature,
\begin{equation}\label{BOsolution}
\zeta_n(t)=\zeta_n(0)\, 
e^{i\omega_n(\zeta(0)) t}\ ,\quad t\in \R , \quad n \ge 1\, .
\end{equation}
By \cite[Proposition 8.1]{GK}),  
$\mathcal H_B := \mathcal H \circ \Phi^{-1} $
can be computed as
\begin{equation}\label{Hamiltonian in Birkhoff}
\mathcal H_B (\zeta) := \sum_{k=1}^\infty k^2 |\zeta_k|^2 -
\sum_{k=1}^\infty (\sum_{p = k }^\infty |\zeta_p|^2 )^2\, ,
\end{equation}
implying that the frequencies, defined by 
\eqref{frequencies in Birkhoff 0}, are given by
\begin{equation}\label{frequencies in Birkhoff}
\omega_n(\zeta) = n^2 - 
2 \sum_{k=1}^{\infty} \min(n,k) |\zeta_k|^2 \, ,
\quad \forall n \ge 1\, .
\end{equation}
Remarkably, for any $n \ge 1$, $\omega_n$ depends 
{\em linearly} on the actions 
$|\zeta_k|^2$, $k \ge 1$. Furthermore, while the Hamiltonian
$\mathcal H_B$ is defined on $h^{1}_+$, 
the frequencies $\omega_n$, $n \ge 1$, given by
\eqref{frequencies in Birkhoff} for $\zeta \in h^1_+,$
extend to bounded functionals on $\ell^2_+$,
\begin{equation}\label{frequency map}
\omega_n: \ell^2_+ \to \R,\, 
\zeta = (\zeta_k)_{k \ge 1} \mapsto \omega_n(\zeta)\, .
\end{equation}
We will prove 
that the restriction $\mathcal S_0$ 
of the solution map of \eqref{BO}
to $L^2_{r,0}$, when expressed in
Birkhoff coordinates,
$$
\mathcal S_B : h^{1/2}_+ \to C(\R, h^{1/2}_+) \, , \
\zeta(0) \mapsto 
(\zeta_n(0)\, e^{i\omega_n(\zeta(0)) t})_{n \ge 1}
$$
is continuous -- see Proposition \ref{S in Birkhoff continuous} in Section \ref{mathcal S_B}.
By Theorem \ref{main result},
$\Phi: L^2_{r,0} \to h^{1/2}_+$ and
its inverse $\Phi^{-1}: h^{1/2}_+ \to L^2_{r,0}$
are continuous. Since
$$
\mathcal S_0 = \Phi^{-1} \mathcal S_B \Phi\, : \,  
L^2_{r,0} \to C(\R, L^2_{r,0}) \, , \, 
u(0) \mapsto  \Phi^{-1}\mathcal S_B(t, \Phi(u(0)))
$$
it follows that 
$\mathcal S_0 : L^2_{r,0} \to C(\R, L^2_{r,0})$
is continuous as well. We remark that for any
$u(0) \in L^2_{r,0}$, the solution
$t \mapsto \mathcal S(t, u(0))$ 
can be approximated
in $L^2_{r,0}$ by classical solutions of equation \eqref{BO} (cf. Remark \ref{RemarkThm1}$(i)$)
and thus coincides with the solution,
obtained by Molinet in \cite{Mol} (cf. also \cite{MP}). 

Starting point of the proof of Theorem \ref{Theorem 1}
is formula \eqref{flow map BO} in Subsection \ref{solution map Sc}. We will show that it
extends to the Sobolev spaces $H^{-s}_{r,0}$ for
any $0 < s < 1/2$.
A key ingredient to prove Theorem \ref{Theorem 1} is
therefore the following result on the extension 
of the Birkhoff map  $\Phi$ to 
$H^{-s}_{r,0}$ for any $0 < s < 1/2$:
\begin{theorem}\label{extension Phi}
{\sc (Extension of $\Phi$.)}
For any $0 < s < 1/2,$ the map $\Phi$ of 
Theorem \ref{main result} admits an extension, 
also denoted by $\Phi$,
$$\Phi :H^{-s}_{r,0} \rightarrow h^{1/2-s}_+, \,\,  u \mapsto 
\Phi (u):=(\zeta_n(u))_{n\ge 1}\, ,$$
so that the following holds: \\
$(i)$ $\Phi$ is a homeomorphism. \\
$(ii)$ There exists an increasing function 
$F_s : \R_{>0} \to \R_{>0}$ so that 
$$
\|u\|_{-s} \le F_s( \| \Phi(u) \|_{1/2 - s}) 
\, , \qquad \forall u \in H^{-s}_{r,0} \, .
$$
$(iii)$ $\Phi$ and its inverse map
bounded subsets to bounded subsets.
\end{theorem}
\begin{remark}\label{weakPhi}
Notice that $(i)$ and $(iii)$ combined with the Rellich compactness theorem imply that for $0\leq s< \frac 12$,
the map $\Phi :H^{-s}_{r,0} \rightarrow h^{1/2-s}_+$ and its inverse $\Phi^{-1}:  h^{1/2-s}_+ \rightarrow  H^{-s}_{r,0}$
 are weakly sequentially continuous on $H^{-s}_{r,0}$.
\end{remark}
\begin{remark}\label{def integral I_s}
The above a priori bound for $\| u \|_{-s}$
can be extended to the space $H^{-s}_r$ as follows
$$
 \| v \|_{-s} 
\le F_s(\| \Phi(v - [v]) \|_{1/2 - s}) + |[v]| \, ,
\quad [v] = \langle v | 1 \rangle \, , \qquad
\forall v \in H^{-s}_r \, .
$$
For any $0 < s < 1/2$, the integral $I_s$ in 
Theorem \ref{Theorem 1}$(iv)$ is defined as
$$
I_s(v) := F_s(\| \Phi(v - [v]) \|_{1/2 - s}) + |[v]| \, .
$$
\end{remark}

\medskip

\noindent
{\em Ideas of the proof of Theorem \ref{extension Phi}.} 
At the heart of the proof of Theorem 1 in \cite{GK} is the Lax operator $L_u$,
appearing in the Lax pair formulation in \cite{Nak} (cf. also \cite{BK}, \cite{CW}, \cite{FA})
$$
\partial_t L_u = [B_u, L_u]
$$
of \eqref{BO} -- see \cite[Appendix A]{GK} for a review.
For any given $u \in L^2_{r}$, the operator $L_u$
is the first order operator acting on the Hardy space $H_+$,
$$
L_u := -i \partial_x  - T_u\, , \qquad 
T_u (\cdot) := \Pi(u \, \cdot)
$$
where $\Pi$ is the orthogonal projector of $L^2$ onto $H_+$ and  $T_u$ is the Toeplitz operator with symbol $u$, 
$$
H_+ := 
\{ f \in L^2 \ : \ \widehat f(n) = 0 \, \, \, \forall n < 0  \} \, .
$$
The operator $L_u$ is self-adjoint with domain 
$H^1_+:= H^1 \cap H_+$, bounded from below, and has a compact resolvent. Its spectrum consists of real eigenvalues which
bounded from below.
When listed in increasing order 
they form a sequence, satisfying
$$
\lambda_0 \le \lambda_1 \le \cdots \, ,   \qquad
\lim_{n \to \infty}\lambda_n = \infty \, .
$$
For our purposes, the most important properties of 
the spectrum of $L_u$ are that 
the eigenvalues are conserved along the flow of \eqref{BO}
and that they are all simple. More precisely, one has
\begin{equation}\label{def nth gap}
\gamma_n := \lambda_n - \lambda_{n-1} -1 \ge 0\, , \quad
\forall n \ge 1\, .
\end{equation}
The nonnegative number $\gamma_n$ is referred to  
as the $n$th gap
of the spectrum ${\rm{spec}}(L_u)$ of $L_u$.
-- see \cite[Appendix C]{GK} for an explanation of this terminology. For any $n \ge 1$, the complex 
Birkhoff coordinate $\zeta_n$
of Theorem \ref{main result} is related to $\gamma_n$ by
$|\zeta_n|^2 = \gamma_n$ whereas its phase is defined
in terms of an  appropriately normalized eigenfunction $f_n$
of $L_u$, corresponding to the eigenvalue $\lambda_n$.\\
A key step for the proof of Theorem \ref{extension Phi} 
is to show that
for any $u \in H^{-s}_{r}$ with $0 < s < 1/2$, 
the Lax operator $L_u$
can be defined as a self-adjoint operator with domain included in
$H^{1-s}_+$ and that its spectrum has properties
similar to the ones described above in the case where
$u \in L^2_r$. In particular, the inequality \eqref{def nth gap} continues to hold.
Since  the proof of Theorem \ref{extension Phi} requires several steps, it
is split up into  two parts, corresponding to
Section \ref{extension of Phi, part 1} and 
Section \ref{extension of Phi, part 2}.

\smallskip

A straightforward application of Theorem \ref{extension Phi}
is the following result on isospectral potentials.
To state it, we need to introduce some additional notation. For any $\zeta \in h^{1/2-s}_+$, define
\begin{equation}\label{def tor}
{\rm{Tor}}(\zeta) := \{ z \in h^{1/2-s}_+ \, : \,
|z_n| = |\zeta_n| \, \,\,\forall n \ge 1 \}.
\end{equation}
Note that ${\rm{Tor}}(\zeta)$ is an infinite product of (possibly degenerate) circles
and a compact
subset of $h^{1/2-s}_+$. Furthermore, for any 
$u \in H^{-s}_{r,0}$, let
$${\rm{Iso}}(u) := \{ v \in H^{-s}_{r,0} \, : \,
 {\rm{spec}}(L_v) = {\rm{spec}}(L_u) \} \, .
$$ 
where as above, ${\rm{spec}}(L_u)$ denotes the spectrum
of the Lax operator $L_u := -i \partial_x  - T_u$.
The spectrum of $L_u$ continues to be
characterized in terms of its gaps $\gamma_n$, $n \ge 1$, 
(cf. \eqref{def nth gap}) and
the extended Birkhoff coordinates continue to satisfy
$| \zeta_n |^2= \gamma_n$, $n \ge 1$.
An immediate consequence of Theorem \ref{extension Phi}
then is that \cite[Corollary 8.1]{GK} extends as follows:
\begin{corollary}\label{isospectral set}
For any $u \in H^{-s}_{r,0}$ with $0 < s < 1/2$, 
$$
\Phi({\rm{Iso}}(u)) = {\rm{Tor}}(\Phi(u)) \, .
$$
Hence by the continuity of $\Phi^{-1}$,
${\rm{Iso}}(u)$ is a compact, connected subset
of $H^{-s}_{r,0}$.
\end{corollary}


\section[extension of Birkhoff map. Part 1]{Extension of $\Phi$. Part 1}\label{extension of Phi, part 1}

In this section we prove the first part of 
Theorem \ref{extension Phi}, which we state
as a separate result:
\begin{proposition}\label{extension Phi, part 1}
{\sc (Extension of $\Phi$. Part 1)}
For any $0 < s < 1/2$, the following holds:\\
$(i)$ For any $n \ge 1$, the formula in \cite[(4.1)]{GK} of
the Birkhoff coordinate
$\zeta_n : L^2_{r,0} \to \C$ extends to $H^{-s}_{r,0}$ 
and for any $u \in H^{-s}_{r,0}$,
$(\zeta_n(u))_{n \ge 1}$ is in $h^{1/2-s}_+$.
The extension of the map $\Phi$ of Theorem \ref{main result}, 
also denoted by $\Phi$,
$$\Phi :H^{-s}_{r,0} \rightarrow h^{1/2-s}_+, \,\,  u \mapsto 
\Phi (u):=(\zeta_n(u))_{n\ge 1}\, ,$$
 maps bounded subsets of $H^{-s}_{r,0}$ 
to bounded subsets of $h^{1/2-s}_+$.\\
$(ii)$ $\Phi$ is sequentially weakly continuous
and one-to-one.
\end{proposition}

 First we need to establish some auxiliary results
 related to the Lax operator $L_u$.

 \begin{lemma}\label{estimate for T_u f}
 Let $u \in H^{-s}_{r,0}$ with $0 \le s < 1/2$.
 Then for any $f, g \in H^{1/2}_+$, 
 the following estimates hold:\\
 $(i)$ There exists a constant $C_{1,s} >0$ only depending
 on $s$, so that
 \begin{equation}\label{para}
 \| f g \|_s \le 
  C_{1,s}^2 \| f\|_\sigma \| g \|_\sigma\, , \qquad
 \sigma:= (1/2 +s)/2 \,.
 \end{equation}
$(ii)$ The expression $\langle  u | f \overline f \rangle$
is well defined and
satisfies the estimate
 \begin{equation}\label{estimate Toeplitz}
 | \langle  u | f \overline f \rangle |
 \le \frac{1}{2} \| f \|_{1/2}^2 + 
 \eta_s(\| u \|_{-s}) \|f \|^2
 \end{equation}
 where
 \begin{equation}\label{def eta_s}
 \eta_s(\| u \|_{-s}) := \|u\|_{-s}
 \big( 2 ( 1 + \|u\|_{-s}))^{\alpha} C_{2,s}^2 \, ,
 \quad \alpha:= \frac{1+2s}{1-2s} \, 
 \end{equation}
 and $C_{2,s} >0$ is a constant, only depending on $s$.
 \end{lemma}
 \begin{proof}
 $(i)$ Estimate \eqref{para} is obtained from standard
 estimates of paramultiplication
 (cf. e.g. \cite[Exercise II.A.5]{AG}, 
 \cite[Theorem 2.82, Theorem 2.85]{BCD}).
 $(ii)$ By item $(i)$,
 $\langle  u | f \overline f \rangle$
 is well defined by duality and satisfies
 $$
 | \langle  u | f \overline f \rangle |
 \le \| u \|_{-s} \|f \overline f \|_s 
 \le  \| u \|_{-s} C_{1,s}^2 \| f\|_\sigma^2 \, .
 $$
 In order to estimate $\| f\|_\sigma^2$, note that 
 by interpolation one has  
$ \| f\|_\sigma \le \| f \|_{1/2}^{1/2 +s}
  \| f \|^{1/2 -s}$ and hence 
 \begin{equation}\label{interpolation}
  C_{1,s} \| f\|_\sigma \le \| f \|_{1/2}^{1/2 +s}
  \big( C_{2,s} \| f \| \big)^{1/2 -s}
 \end{equation}
 for some constant $C_{2,s}> 0$.
 Young's inequality then yields
 for any $\e > 0$
 \begin{equation}\label{Young}
 \big( C_{1,s} \| f\|_\sigma \big)^2 \le
 \e \| f \|_{1/2}^2 + 
 \e^{-\alpha} \big( C_{2,s} \| f \| \big)^2\, ,
 \qquad \alpha = \frac{1+2s}{1-2s} \, .
 \end{equation}
 Estimate \eqref{estimate Toeplitz} then
 follows from \eqref{Young} by
 choosing $\e = \big( 2 ( 1 + \|u\|_{-s}))^{-1}$. 
 \end{proof}
 Note that estimate \eqref{estimate Toeplitz} implies
 that the sesquilinear form $\langle T_u f | g \rangle $
 on $H^{1/2}_+$, 
 obtained from the Toeplitz operator $T_u f := \Pi (u f)$
 with symbol $u \in L^2_{r,0}$,
 can be defined  for any $u \in H^{-s}_{r,0}$ 
 with $0 \le s < 1/2$ by 
setting 
$\langle T_u f | g \rangle := 
\langle u | g \overline f  \rangle$ and that it is bounded.
 For any $u \in H^{-s}_{r,0},$ we then define the 
 sesquilinear form $Q_u^+$ on $H^{1/2}_+$ as follows
 \begin{equation}\label{Q_u^+}
 Q_u^+(f, g) : = \langle -i\partial_x f | g \rangle
 - \langle T_u f | g \rangle +
 \big( 1 + \eta_s(\|u\|_{-s}) \big) \langle f | g \rangle
 \end{equation}
 where $\eta_s(\|u\|_{-s})$ is given by \eqref{def eta_s}. 
 The following lemma says that the quadratic form 
 $Q^+_u(f, f)$ is equivalent to $\|f\|_{1/2}^2$.
 More precisely, the following holds.
 \begin{lemma}\label{comparison Q_u^+}
 For any $ u \in H^{-s}_{r,0}$ with $0 \le s < 1/2$,
 $Q^+_u$ is a positive, sesquilinear form,
 satisfying 
 $$
 \frac{1}{2} \| f \|_{1/2}^2 \le  Q^+_u(f, f) \le 
 \big( 3 + 2 \eta_s(\|u\|_{-s}) \big) \| f \|_{1/2}^2 \, ,
 \quad \forall f \in H^{1/2}_+\, .
 $$
 \end{lemma}
 \begin{proof}
 (i) Using that $u$ is real valued, one verifies that
 $Q_u^+$ is sesquilinear.
 The claimed estimates are obtained from 
 \eqref{estimate Toeplitz} as follows:
 since $\langle n \rangle \le 1 + |n|$
 one has
 $
 \| f \|_{1/2}^2 \le 
 \langle -i \partial_x f | f \rangle  + \| f \|^2$,
 and hence by \eqref{estimate Toeplitz},
 $$
 | \langle T_u f | f \rangle | \le 
 \frac{1}{2} \langle -i \partial_x f | f \rangle
 + \big(\frac{1}{2} + \eta_s(\|u\|_{-s}) \big)\| f \|^2 \, .
 $$
 By the definition \eqref{Q_u^+}, the claimed estimates 
 then follow.
 In particular, the lower bound for $Q_u^+(f, f)$ shows that
 $Q_u^+$ is positive.
 \end{proof}
 Denote by 
 $\langle f | g \rangle_{1/2} \equiv \langle f | g \rangle_{H^{1/2}_+}$
 the inner product, corresponding to the norm 
 $\| f \|_{1/2}$. It is given by
 $$
 \langle f |  g \rangle_{1/2} =
 \sum_{n \ge 0} \langle n \rangle 
 \widehat f(n) \overline{\widehat g (n)}
 \, , \quad \forall f, g \in H^{1/2}_+ \, .
 $$
 Furthermore, denote by 
 $D : H^t_+ \to H^{t-1}_+$ and
 $\langle D \rangle : H^t_+ \to H^{t-1}_+$, $t\in \R$, 
 the Fourier multipliers, defined for $f \in H^t_+$
 with Fourier series
 $f = \sum_{n= 0}^\infty \widehat f(n) e^{inx}$ by 
 $$
  D f := -i\partial_x f = \sum_{n = 0}^\infty 
 n  \widehat f(n) e^{inx} \, , \qquad
 \langle D \rangle f := \sum_{n = 0}^\infty 
 \langle n \rangle \widehat f(n) e^{inx} \, .
 $$
 \begin{lemma}\label{Lax Milgram}
 For any $ u \in H^{-s}_{r,0}$ with $0 \le s < 1/2$,
 there exists a bounded linear isomorphism 
 $A_u : H^{1/2}_+ \to H^{1/2}_+$ so that
 $$
 \langle A_u f | g \rangle_{1/2} =
 Q^+_u(f, g)\, , \quad \forall f, g \in H^{1/2}_+\, .
 $$
 The operator $A_u$ has the following properties:\\
 $(i)$ $A_u$ and its inverse $A_u^{-1}$ are symmetric,
 i.e., for any $f, g \in H^{1/2}_+,$
 $$
 \langle A_u f | g \rangle_{1/2} =
 \langle  f | A_u g \rangle_{1/2}\, , \quad
 \langle A_u^{-1} f | g \rangle_{1/2} =
 \langle  f | A_u^{-1}g \rangle_{1/2} \, .
 $$
 $(ii)$ The linear isomorphism $B_u$, given by the composition
 $$
 B_u := \langle D \rangle A_u : H^{1/2}_+ \to H^{-1/2}_+
 $$
 satisfies
 $$
Q_u^+ (f, g) = \langle B_u f | g \rangle \, , 
\quad \forall f, g \in H^{1/2}_+ \, .
 $$
 The operator norm of 
 $B_u$ and the one of its inverse can be bounded uniformly
 on bounded subsets of elements $u$ in $H^{-s}_{r,0}$.
 \end{lemma}
 \begin{proof}
 By Lemma \ref{comparison Q_u^+}, the 
 sesquilinear form
 $Q_u^+$ is an inner product on $H^{1/2}_+$, equivalent to
 the inner product $\langle \cdot | \cdot \rangle_{1/2}$.
 Hence by the theorem of Fr\'echet-Riesz, for any 
 $g \in H^{1/2}_+$, there exists a unique element in 
 $H^{1/2}_+$, which we denote by $A_ug$, so that
 $$
 \langle A_u g | f \rangle_{1/2} =
 Q^+_u(g, f)\, , \quad  \forall f \in H^{1/2}_+\, .
 $$Invitation for special issue in honor of Tony Bloch in Journal  of Geometric Mechanics
 Then $A_u : H^{1/2}_+ \to H^{1/2}_+$ is a linear, injective
 operator, which by Lemma \ref{comparison Q_u^+} is bounded,
 i.e., for any $f, g \in H^{1/2}_+$,
 \begin{align}
 | \langle A_u g |  f \rangle_{1/2} \ | & =
 | Q^+_u(g, f) | \le 
 Q^+_u(g, g)^{1/2} Q^+_u(f, f)^{1/2} \nonumber \\
 & \le \big( 3 + 2 \eta_s(\|u\|_{-s}) \big) 
 \|g\|_{1/2} \| f \|_{1/2} \nonumber \, ,
\end{align}
implying that 
$\| A_u g \|_{1/2} \le 
\big( 3 + 2 \eta_s(\|u\|_{-s}) \big) \| g \|_{1/2}$.

Similarly, by the theorem of Fr\'echet-Riesz, for any 
 $h \in H^{1/2}_+$, there exists a unique element in 
 $H^{1/2}_+$, which we denote by $E_uh$, so that
$$
 \langle h | f \rangle_{1/2} =
 Q^+_u(E_u h, f)\, , \quad  \forall f \in H^{1/2}_+\, .
 $$
 Then $E_u : H^{1/2}_+ \to H^{1/2}_+$ is a linear, injective
 operator, which by Lemma \ref{comparison Q_u^+} is bounded, i.e.,
 $$
 \frac{1}{2} \|E_u h\|_{1/2}^2 \le 
  Q^+_u(E_u h, E_u h) =
  \langle h | E_u h \rangle_{1/2} \le 
  \|h \|_{1/2} \| E_u h \|_{1/2}\, ,
 $$
 implying that  
 $\|E_u h\|_{1/2} \le 2 \| h\|_{1/2}$.
Note that $A_u(E_u h) = h$ and hence $E_u$ is the inverse
of $A_u$. Therefore, $A_u : H^{1/2}_+ \to H^{1/2}_+$
is a bounded linear isomorphism. Next we show item $(i)$.
For any $f, g \in H^{1/2}_+$,
$$
 \langle g | A_u f \rangle_{1/2} =
 \overline{ \langle A_u f |  g \rangle}_{1/2} =
 \overline{Q_u^+ (f, g)} =
 Q_u^+(g, f) = \langle A_u g | f \rangle_{1/2} \, .
$$
The symmetry of $A_u^{-1}$ is proved in the same way.
Towards item $(ii)$, note that for any $f, g \in H^{1/2}_+,$
$\langle f |  g \rangle_{1/2} = 
\langle \langle D \rangle f | g \rangle$
and therefore
$$
 \langle A_u g |  f \rangle_{1/2} =
  \langle \langle D \rangle A_u g |  f \rangle\, ,
$$
implying that the operator 
$B_u = \langle D \rangle A_u : H^{1/2}_+ \to H^{-1/2}_+$  
is a bounded linear isomorphism and that
$$
\langle B_u g |  f \rangle = Q_u^+(g, f) \, , \quad
\forall g,f \in H^{1/2}_+ \, . 
$$
The last statement of (ii) follows from Lemma \ref{comparison Q_u^+}.
\end{proof}

We denote by $L_u^+$ the restriction
 of $B_u$ to ${\rm{dom}}(L_u^+)$, defined as
 $$
 {\rm{dom}}(L_u^+) := \{ g \in H^{1/2}_+ \, : \,  
 B_u g \in H_+\} \, .
 $$
 We view $L^+_u$ as an unbounded linear operator on $H_+$
 and write
 $L^+_u : {\rm{dom}}( L_u^+) \to H_+$.
 \begin{lemma}\label{operator L^+_u}
 For any $ u \in H^{-s}_{r,0}$ with $0 \le s < 1/2$,
 the following holds:\\
 $(i)$ ${\rm{dom}}(L_u^+)$ is a dense subspace of $H^{1/2}_+$
 and hence of $H_+$.\\
 $(ii)$ $L_u^+ : {\rm{dom}}( L_u^+) \to H_+$ is bijective
 and the right inverse of $L^+_u$,
 $(L^+_u)^{-1} : H_+ \to H_+$,
 is compact. Hence $L_u^+$ has discrete spectrum.\\
 $(iii)$ $(L^+_u)^{-1}$ is symmetric and $L^+_u$ is self-adjoint
 and positive.
 \end{lemma}
 \begin{proof}
 $(i)$ Since $H_+$ is a dense subspace of $H^{-1/2}_+$
 and $B_u^{-1} : H^{-1/2}_+ \to H^{1/2}_+$ 
 is a linear isomorphism, 
 ${\rm{dom}}(L_u^+) = B_u^{-1}(H_+)$ is a dense subspace
 of $H^{1/2}_+$, and hence also of $H_+$.\\
 $(ii)$ Since $L^+_u$ is the restriction of 
 the linear isomorphism $B_u$,
 it is one-to-one. By the definition of $L^+_u$, it is onto.
 The right inverse of $L_u^+$, denoted by $(L_u^+)^{-1}$, 
 is given by the composition
 $\iota \circ B_u^{-1}|_{H_+}$, where
 $\iota: H^{1/2}_+ \to H_+$ is the standard embedding
 which by Sobolev's embedding theorem is compact.
 It then follows that $(L_u^+)^{-1}: H_+ \to H_+$
 is compact as well.\\
 $(iii)$ For any $f, g \in H_+$
 $$
 \langle (L_u^+)^{-1} f | g \rangle =
 \langle A_u^{-1}\langle D \rangle ^{-1}  f | g \rangle =
 \langle A_u^{-1}\langle D  \rangle ^{-1}  f | 
 \langle D  \rangle ^{-1} g \rangle_{1/2} \, .
 $$
 By Lemma \ref{Lax Milgram}, $A_u^{-1}$ is symmetric with respect to the
 $H^{1/2}_+-$inner product. Hence 
 $$
 \langle (L_u^+)^{-1} f | g \rangle =
 \langle \langle D  \rangle ^{-1}  f | 
 A_u^{-1} \langle D \rangle ^{-1} g \rangle_{1/2}
 =  \langle f | (L_u^+)^{-1} g \rangle\, ,
 $$
 showing that $(L_u^+)^{-1}$ is symmetric.
 Since in addition, $(L_u^+)^{-1}$ is bounded
 it is also self-adjoint.
 By Lemma \ref{comparison Q_u^+} it then follows that
 $$
  \langle L_u^+ f | f \rangle = 
 \langle  \langle D \rangle A_u f | f \rangle=
 \langle A_u f | f \rangle_{1/2}=
 Q^+_u(f, f) \ge \frac{1}{2} \| f \|_{1/2}^2\, ,
 $$
 implying that $ L_u^+$ is a positive operator.
 \end{proof}
 We now define for any $ u \in H^{-s}_{r,0}$ 
 with $0 \le s < 1/2$, the operator $L_u$ 
 as a linear operator with domain 
 ${\rm{dom}}(L_u) := {\rm{dom}}(L_u^+)$
 by setting
 $$
 L_u := L_u^+ - \big( 1 + \eta_s(\|u\|_s) \big) :
 {\rm{dom}}(L_u) \to H_+ \, .
 $$
 Lemma \ref{operator L^+_u} yields the following
 \begin{corollary}\label{definition L_u}
 For any $ u \in H^{-s}_{r,0}$ with $0 \le s < 1/2$,
 the operator $L_u : {\rm{dom}}(L_u) \to H_+$
 is densely defined, self-adjoint, bounded from below,
 and has discrete spectrum. It thus admits an $L^2-$normalized
 basis of eigenfunctions, contained in ${\rm{dom}}(L_u)$
 and hence in $H^{1/2}_+$.
 \end{corollary}
 \begin{remark}\label{symmetry of B_u}
 Let $ u \in H^{-s}_{r,0}$ with $0 \le s < 1/2$ be given.
 Since ${\rm{dom}}(L^+_u)$ is dense in $H^{1/2}_+$
 and $L^+_u$ is the restriction of 
 $B_u : H^{1/2}_+ \to H^{-1/2}_+$ to ${\rm{dom}}(L^+_u)$,
 the symmetry
 $$
 \langle L_u^+  f | g \rangle = \langle f | L_u^+  g \rangle 
 \, , \quad \forall f, g \in \rm{dom}(L^+_u)
 $$
 can be extended by a straightforward density argument
 as follows
 $$
  \langle B_u  f | g \rangle = \langle f | B_u  g \rangle 
 \, , \quad \forall f, g \in H^{1/2}_+ \, .
 $$
 \end{remark}
Note that for any $f, g \in H^{1/2}_+$, 
$\langle B_u  f | g \rangle =
\langle \langle D \rangle A_u f | g \rangle =
\langle A_u f | g \rangle_{1/2}$ 
and hence by \eqref{Q_u^+},
$$
\langle B_u  f | g \rangle =
Q_u^+(f, g) = 
\langle Df - T_uf + (1 + \eta_s(\|u\|_{-s}))f | g \rangle \, ,
$$ 
yielding the following identity in $H^{-1/2}_+$,
\begin{equation}\label{identity B_u}
B_u  f = Df - T_uf + (1 + \eta_s(\|u\|_{-s})) f\, ,
\quad \forall f \in H^{1/2}_+ \, .
\end{equation}
 Given $ u \in H^{-s}_{r,0}$ with $0 \le s < 1/2$, 
 let us consider the restriction of $B_u$ to $H^{1-s}_+$.
 \begin{lemma}\label{restriction of B_u to H^{1-s}}
 Let $ u \in H^{-s}_{r,0}$ with $0 \le s < 1/2$.
 Then $B_u(H^{1-s}_+) = H^{-s}_+$ and the 
 restriction
 $B_{u; 1-s}  :=B_u|_{H^{1-s}_+} : H^{1-s}_+ \to H^{-s}_+$
 is a linear isomorphism. The operator norm of
 $B_u|_{H^{1-s}_+}$ 
 and the one of its inverse are bounded uniformly
 on bounded subsets of elements $u \in H^{-s}_{r,0}$.
 \end{lemma}
 \begin{proof}
 Since $1-s>1/2$, $H^{1-s}$ acts by multiplication on itself and on $L^2$, hence  by interpolation on $H^r$ for $0\leq r\leq 1-s$. 
 By duality, it also acts on $H^{-r}$, in particular 
 with $r=s$. This implies that  $B_u|_{H_+^{1-s}} : H_+^{1-s} \to H_+^{-s}$ is bounded.
 Being the restriction of an injective operator, it is
 injective as well. Let us prove that $B_u|_{H_+^{1-s}}$ has
 $H_+^{-s}$ as its image. To this end consider an arbitrary
 element $h \in H_+^{-s}$. We need to show that the
 solution $f\in H^{1/2}_+$ of $B_u f = h$ 
 is actually in $H^{1-s}_+$. Write
 \begin{equation}\label{identity for f}
 Df  = h +  (1 + \eta_s(\|u\|_{-s})) f + T_uf\, .
 \end{equation}
 Note that $h +  (1 + \eta_s(\|u\|_{-s})) f$ is in $H^{-s}_+$
 and it remains to study $T_uf$.
 By Lemma \ref{estimate for T_u f}$(i)$ one infers
 that for any $g \in H^\sigma_+$, with $\sigma =(1/2+s)/2$, 
 $$
| \langle T_u f | g \rangle | = 
| \langle u  | g \overline f \rangle | \le 
\|u \|_{-s} \|g \overline f\|_{s} \le 
C_{1,s}^2 \|u\|_{-s} \| g \|_{\sigma} \| f \|_{\sigma}\, ,
 $$
 implying that $T_u f \in H^{-\sigma}_+$ and hence by
 \eqref{identity for f}, $f \in H^{1-\sigma}$.
 Since $1-\sigma >1/2$,  we argue as at the beginning of the proof 
to infer that $T_u f \in H^{-s}_+$. Thus applying
 \eqref{identity for f} once more  we conclude that 
  $f \in H^{1-s}_+$. This shows that 
  $B_u|_{H_+^{1-s}} : H_+^{1-s} \to H_+^{-s}$ is onto.
  Going through the arguments of the proof one verifies
  that the operator norm of $B_u|_{H^{1-s}_+}$ 
 and the one of its inverse are bounded uniformly
 on bounded subsets of elements $u \in H^{-s}_{r,0}$.
  This completes the proof of the lemma.
 \end{proof}
 Lemma \ref{restriction of B_u to H^{1-s}} has
 the following important 
 \begin{corollary}
 For any $ u \in H^{-s}_{r,0}$ with $0 \le s < 1/2$,
 ${\rm{dom}}(L_u^+) \subset H^{1-s}_+$. In particular,
 any eigenfunction of $L_u^+$ (and hence of $L_u$)
  is in $H^{1-s}_+$.
 \end{corollary}
 \begin{proof}
 Since $H_+ \subset H^{-s}_+,$ one has 
 $B_u^{-1}(H_+) \subset B_u^{-1}(H^{-s}_+)$
 and hence by Lemma \ref{restriction of B_u to H^{1-s}},
${\rm{dom}}(L_u^+) = B_u^{-1}(H_+) 
\subset  H^{1-s}_+$.
\end{proof}
With the results obtained so far, it is straightforward 
to verify that many of the results of \cite{GK}
extend to the case where $u \in H^{-s}_{r,0}$.
More precisely, let $u \in H^{-s}_{r,0}$ with $0 \le s < 1/2$.
We already know that the spectrum of $L_u$ is discrete, 
bounded from below, and real. When listed 
in increasing order and with their multiplicities,
the eigenvalues of $L_u$ satisfy
$\lambda_0 \le \lambda_1 \le \lambda_2 \le \cdots $.
Arguing as in the proof of \cite[Proposition 2.1]{GK}, 
one verifies that $\lambda_n \ge \lambda_{n-1}+ 1$, $n \ge 1$,
and following \cite[(2.10)]{GK} we define
$$
\gamma_n(u) := \lambda_n - \lambda_{n-1} - 1 \ge 0\, .
$$
It then follows that for any $n \ge 1$,
$$
\lambda_n  = n + \lambda_0 + 
\sum_{k=1}^n\gamma_k \ge n + \lambda_0\, .
$$
Since \cite[Lemma 2.1, Lemma 2.2]{GK}
continue to hold for $u \in H^{-s}_{r,0}$, 
we can introduce eigenfunctions $f_n(x, u)$
of $L_u$, corresponding to the eigenvalues $\lambda_n$, which
are normalized as in \cite[Definition 2.1]{GK}. 
The identities \cite[(2.13)]{GK} continue to hold, 
\begin{equation}\label{formula coeff of u}
\lambda_n \langle 1 | f_n \rangle = - \langle u | f_n \rangle
\end{equation}
as does \cite[Lemma 2.4]{GK}, stating that $\gamma_n = 0$
if and only if $\langle 1 | f_n \rangle = 0.$
Furthermore, the definition \cite[(3.1)]{GK} of the
generating function $\mathcal H_\lambda(u)$
extends to the case where $u \in H^{-s}_{r,0}$
with $0 < s < 1/2$,
\begin{equation}\label{definition generation fucntion}
 \mathcal H_\lambda: H^{-s}_{r,0} \to \C, \, u \mapsto 
\langle (L_u + \lambda)^{-1} 1 | 1 \rangle \, .
\end{equation}
and so do the identity \cite[(3.2)]{GK},
the product representation of $\mathcal H_\lambda(u)$,
stated in \cite[Proposition 3.1(i)]{GK}, 
\begin{equation}\label{generating function for -s}
\mathcal H_\lambda (u)=\frac{1}{\lambda_0+\lambda}\prod_{n=1}^\infty \big(1-\frac{\gamma_n}{\lambda_n+\lambda}   \big) \, ,
\end{equation}
and the one for $|\langle 1 | f_n \rangle |^2$,
$n \ge 1$, given in \cite[Corollary 3.1]{GK}, 
 \begin{equation}\label{product for kappa for -s}
 | \langle 1 | f_n \rangle |^2 = \gamma_n \kappa_n\, , \quad
\kappa_n = \frac{1}{\lambda_n - \lambda_0}
\prod_{p \ne n} (1 - \frac{\gamma_p}{\lambda_p - \lambda_n})
\, .
\end{equation}
The product representation 
\eqref{generating function for -s} then
yields the identity (cf. \cite[Proposition 3.1(ii)]{GK}
and its proof),
\begin{equation}\label{identity for lambda_0}
-\lambda_0(u) = \sum_{n=1}^\infty \gamma_n(u) \,.
\end{equation}
Since $\gamma_n(u) \ge 0$ for any $n \ge 1,$ one infers that
for any $u \in H^{-s}_{r,0}$ with $0 \le s < 1/2$,
the sequence
$(\gamma_n(u))_{n \ge 1}$ is in 
$\ell^1_+ \equiv \ell^1(\N , \R)$ and
\begin{equation}\label{formula for lambda_n}
\lambda_n(u) = n - \sum_{k = n+1}^\infty \gamma_k(u)
\le n\, .
\end{equation}
By \eqref{Q_u^+}, Lemma \ref{comparison Q_u^+}, and \eqref{identity B_u}, we  infer
that $- \lambda_0 \le \frac 12 + \eta_s(\|u\|_{-s})$, 
yielding, when combined with \eqref{identity for lambda_0} 
and \eqref{formula for lambda_n}, the estimate
\begin{equation}\label{sandwich lambda_n}
 n -  \frac 12 - \eta_s(\|u\|_{-s})  \le \lambda_n(u) \le n \, ,
 \quad \forall n \ge 0\, .
\end{equation}
 
In a next step we want consider the linear isomorphism 
 $$
 B_{u; 1-s} = B_u|_{H_+^{1-s}} : H^{1-s}_+ \to H^{-s}_+
 $$ 
 on the scale of Sobolev spaces. By duality,
 $B_{u; 1 - s}$ extends as a bounded linear isomorphism,
 $B_{u,s}:  H^{s}_+ \to H^{-1+s}_+ $ and hence by
 complex interpolation, for any $s \le t \le 1-s$, 
 the restriction of $B_{u,s}$ to $H^{t}_+$
 gives also rise to a bounded linear isomorphism,
 $B_{u; t}: H^{t}_+ \to H^{-1 +t}_+$.
 All these operators satisfy the same bound as $B_{u; 1 - s}$
 (cf. Lemma \ref{restriction of B_u to H^{1-s}}).
 To state our next result, it is convenient to introduce 
 the notation $\N_0 := \Z_{\ge 0}$. Recall that 
 $h^t(\N_0) = h^t(\N_0, \C)$, $t \in \R$,
 and that we write $\ell^2(\N_0)$ instead of
 $h^0(\N_0)$.
 \begin{lemma}\label{basis on Sobolev scale}
 Let $ u \in H^{-s}_{r,0}$ with $0 \le s < 1/2$
 and let $(f_n)_{n \ge 0}$ be the basis of $L^2_+,$
consisting of eigenfunctions of $L_u$ with
$f_n$, $n \ge 0$,
 corresponding to the eigenvalue $\lambda_n$ and normalized
 as in \cite[Definition 2.1]{GK}. Then for any 
 $-1 + s \le t \le 1-s,$
 $$
 K_{u;t} : H^t_+ \to h^t(\N_0)\, , \, 
 f \mapsto (\langle f | f_n \rangle)_{n \ge 0}
 $$
 is a linear isomorphism. In particular, for 
 $f = \Pi u \in H^{-s}_+$, one obtains that
 $(\langle \Pi u | f_n \rangle)_{n \ge 0} 
 \in h^{-s}(\N_0)$.
 The operator norm of $K_{u;t}$ and the one of its
 inverse can be uniformly bounded 
 for $-1 + s \le t \le 1-s$ and for  $u$ in a bounded subset of $ \in H^{-s}_{r,0}$.
 \end{lemma}
 \begin{proof} We claim that the sequence $(\tilde f_n)_{n\ge 0}$, defined by
 $$
 \tilde f_n=\frac{f_n}{(\lambda_n+1+\eta_s(\Vert u\Vert_{-s}))^{1/2}} \, ,
 $$
 is an orthonormal basis of the Hilbert space $H^{1/2}_+$, endowed with the inner product $Q_u^+$. 
 Indeed, for any $n \ge 0$ and any $g\in H^{1/2}_+$, one has
 $$Q_u^+(\tilde f_n, g) =\langle L_u^+\tilde f_n\vert g\rangle =(\lambda_n+1+\eta_s(\Vert u\Vert_{-s}))^{1/2}\langle f_n\vert g\rangle \ .$$
As a consequence, for any $n, m \ge 0$, $Q_u^+(\tilde f_n, \tilde f_m) = \delta_{nm}$ and 
 the orthogonal complement of the subspace of $H^{1/2}_+$, spanned by  $(\tilde f_n)_{n\ge 0}$, is the trivial vector space $\{ 0\} $,
showing that $(\tilde f_n)_{n\ge 0}$ is an orthonormal basis of $H^{1/2}_+$.
In view of \eqref{sandwich lambda_n}, we then conclude that
  $$
 K_{u;1/2} : H^{1/2}_+ \to h^{1/2}(\N_0)\, , \, 
 f \mapsto (\langle f | f_n \rangle )_{n \ge 0}
  $$ 
  is a linear isomorphism.
  Its inverse is given by
  $$
  K_{u;1/2}^{-1} : h^{1/2}(\N_0) \to  H^{1/2}_+\, , \,
  (z_n)_{n \ge 0} \mapsto f:= \sum_{n=0}^\infty z_n f_n
  \ .$$
 By interpolation we infer that 
  for any $0 \le t \le 1/2$,
  $ K_{u;t} : H^{t}_+ \to h^{t}(\N_0)$
  is a linear isomorphism. Taking the transpose of
  $K_{u;t}^{-1}$ it then follows that for any 
  $0 \le t \le 1/2$,
  $$
   K_{u;-t} : H^{-t}_+ \to h^{-t}(\N_0)\, ,\,
   f \mapsto (\langle f | f_n \rangle )_{n \ge 0}\, ,
  $$
 is also a linear isomorphism. 
 It remains to discuss the remaining range of $t,$
 stated in the lemma.
 By Lemma \ref{restriction of B_u to H^{1-s}}, 
 the restriction of $B_u^{-1}$ to $H^{-s}_+$ 
 gives rise to a linear isomorphism
 $
 B_{u; 1-s}^{-1} :  H^{-s}_+ \to  H^{1-s}_+ \,.
 $
 For any $f \in H^{-s}_+$, one then has
 $$
 B_{u; 1-s}^{-1} f  = \sum_{n=0}^\infty 
 \frac{\langle f | f_n \rangle}
 {\lambda_n + 1 + \eta_s(\|u\|_{-s})} f_n \,.
 $$
 Since by our considerations above, 
 $(\langle f | f_n \rangle )_{n \ge 0} \in h^{-s}(\N_0)$
 one concludes that the sequence
 $\big( \frac{\langle f | f_n \rangle}
 {\lambda_n + 1 + \eta_s(\|u\|_{-s})} \big)_{n \ge 0}$
 is in $h^{1-s}(\N_0)$.
 Conversely, assume that $(z_n)_{n \ge 0} \in h^{1-s}(\N_0)$.
Then 
$\big((\lambda_n + 1 + \eta_s(\|u\|_{-s}))z_n  \big)_{n \ge 0}$ 
is in $h^{-s}(\N_0)$. Hence by the considerations above
on $K_{u; -s}$, there exists $g \in H^{-s}_+$
so that
$$
\langle g | f_n \rangle = 
(\lambda_n + 1 + \eta_s(\|u\|_{-s}))z_n\, , 
\quad \forall n \ge 0\, .
$$
Hence 
$$
g = \sum_{n=0}^\infty 
z_n (\lambda_n + 1 + \eta_s(\|u\|_{-s})) f_n 
= \sum_{n=0}^\infty 
z_n B_{u} f_n
$$ 
and $f:= B_{u}^{-1}g$ is in $H^{1-s}_+$ and satisfies 
$f = \sum_{n=0}^\infty z_n f_n$. Altogether we have thus
proved that
$$
K_{u;1 -s} : H^{1 -s}_+ \to h^{1 -s}(\N_0)\, ,\,
f \mapsto (\langle f | f_n \rangle )_{n \ge 0}\, ,
  $$
is a linear isomorphism. Interpolating between
$K_{u;-s}$ and $K_{u;1 -s}$ and between the adjoints
of their inverses shows that
for any $-1 + s \le t \le 1-s,$
 $$
 K_{u;t} : H^t \to h^t(\N_0)\, , \, 
 f \mapsto (\langle f | f_n \rangle)_{n \ge 0}
 $$
 is a linear isomorphism. Going through the arguments
 of the proof one verifies that 
 the operator norm of $K_{u;t}$ and the one of its
 inverse can be uniformly bounded 
 for $-1 + s \le t \le 1-s$ and for bounded subsets
 of elements $u \in H^{-s}_{r,0}$.
 \end{proof}
 
\smallskip

With these preparations done, we can now prove
Proposition \ref{extension Phi, part 1}$(i)$.

\smallskip
 \noindent
{\em Proof of Proposition \ref{extension Phi, part 1}$(i)$. }
 Let $ u \in H^{-s}_{r,0}$ with $0 \le s < 1/2$.
 By \eqref{product for kappa for -s}, one has 
 for any $n \ge 1$, 
 $$
 | \langle 1 | f_n \rangle |^2 = \gamma_n \kappa_n\, , \quad
\kappa_n = \frac{1}{\lambda_n - \lambda_0}
\prod_{p \ne n} (1 - \frac{\gamma_p}{\lambda_p - \lambda_n})
\, .
$$
Note that the infinite product is absolutely
convergent since the sequence $(\gamma_n(u))_{n \ge 1}$ 
is in $\ell^1_+$ (cf. \eqref{identity for lambda_0}). 
Furthermore, since
$$
1 - \frac{\gamma_p}{\lambda_p - \lambda_n}=
\frac{\lambda_{p-1} + 1 - \lambda_n}{\lambda_p - \lambda_n}
> 0 \, , \quad \forall p \ne n
$$
it follows that $\kappa_n > 0$ for any $n \ge 1$. 
Hence, the formula \cite[(4.1)]{GK} 
of the Birkhoff coordinates $\zeta_n(u)$, $n \ge 1$, 
defined for $u \in L^2_{r,0}$,
\begin{equation}\label{def coo for -s}
\zeta_n(u) = \frac{1}{\sqrt{\kappa_n(u)}} 
\langle 1 | f_n(\cdot, u) \rangle \, ,
\end{equation}
extends to $H^{-s}_{r,0}$. 
By \eqref{formula coeff of u} 
one has (cf. also \cite[(2.13)]{GK})
$$
\lambda_n \langle 1 | f_n \rangle 
= - \langle u | f_n \rangle
= - \langle \Pi u | f_n \rangle \, .
$$
 Since by Lemma \ref{basis on Sobolev scale},
 $(\langle \Pi u | f_n \rangle)_{n \ge 0} 
 \in h^{-s}(\N_0)$
and by \eqref{sandwich lambda_n}
$$
n -  \frac 12 - \eta_s(\|u\|_{-s})  \le \lambda_n(u) \le n \, , \quad \forall n \ge 0 \, ,
$$
 one concludes that 
 $$
( \langle 1 | f_n \rangle)_{n \ge 1} \in h^{1-s}_+\, , \quad
\kappa_n^{-1/2} = \sqrt{n} + o(1)
 $$ 
 and hence $(\zeta_n(u))_{n \ge 1} \in h^{1/2-s}_+$. 
In summary, we have proved that for any $0 < s < 1/2,$
the Birkhoff map $\Phi: L^2_{r,0} \to h^{1/2}_+$ of
Theorem \ref{main result} extends to a map 
$$
H^{-s}_{r,0} \to h^{1/2 - s}_+ \, , \,
u \mapsto  (\zeta_n(u))_{n \ge 1}\, , 
$$
 which we again denote by $\Phi$. Going through the arguments
 of the proof one verifies that $\Phi$ maps bounded subsets of $H^{-s}_{r,0}$ into bounded subsets of $h^{1/2 - s}_+$ . 
 \hfill $\square$

\smallskip

To show Proposition \ref{extension Phi, part 1}$(ii)$
we first need to prove some additional auxilary results.
By \eqref{definition generation fucntion},
the generating function is defined as
$$
 \mathcal H_\lambda: H^{-s}_{r,0} \to \C, \, u \mapsto 
\langle (L_u + \lambda)^{-1} 1 | 1 \rangle \, .
$$
For any given $u \in H^{-s}_{r,0}$, $\mathcal H_\lambda(u)$
is a meromorphic function
in $\lambda \in \C$ with possible poles at the eigenvalues of 
$L_u$ and satisfies (cf.  \eqref{generating function for -s})
\begin{equation}\label{product}
 \mathcal H_\lambda (u)=\frac{1}{\lambda_0+\lambda}\prod_{n=1}^\infty \big(1-\frac{\gamma_n}{\lambda_n+\lambda}   \big)\, . 
\end{equation}

\begin{lemma}\label{gamma_n for -s}
For any $0 \le s < 1/2,$ the following holds:\\
$(i)$ For any $\lambda \in \C \setminus \R,$ 
$\mathcal H_\lambda: H^{-s}_{r,0} \to \C$ is 
sequentially weakly continuous.\\
$(ii)$ $(\sqrt{\gamma_n})_{n \ge 1} : H^{-s}_{r,0} \to 
h^{1/2 -s}_+$
is sequentially weakly continuous. 
In particular, for any $n \ge 0,$
$\lambda_n: H^{-s}_{r,0} \to \R$ is 
sequentially weakly continuous.
\end{lemma}
\begin{proof}
$(i)$ 
Let $(u^{(k)})_{k \ge 1}$ be a sequence in 
$H^{-s}_{r,0}$ with
$u^{(k)}\rightharpoonup u$ weakly in $H^{-s}_{r,0}$ 
as $k \to \infty$. By the definition of $\zeta_n(u)$
(cf. \eqref{product for kappa for -s} -- 
\eqref{def coo for -s})
one has $| \zeta_n(u) |^2 = \gamma_n(u)$.
Since by 
Proposition \ref{extension Phi, part 1}$(i)$,
$\Phi$ maps bounded subsets of $H^{-s}_{r,0}$
to bounded subsets of $h^{1/2 -s}_+$,
there exists $M > 0$ so that for any $k \ge 1$
$$
\|u\|, \,  \|u^{(k)}\| \le M\, \qquad
\sum_{n=1}^\infty n^{1-2s} \gamma_n(u), \,\,
\sum_{n=1}^\infty n^{1-2s} \gamma_n(u^{(k)}) \le M\, .
$$
By passing to a subsequence, if needed, we may assume that
\begin{equation}\label{lim gamma_n for -s}
\big( \gamma_n(u^{(k)})^{1/2}\big)_{n \ge 1}
\rightharpoonup (\rho_n^{1/2})_{n \ge 1}
\end{equation}
weakly in $h^{1/2 -s}(\N, \R)$ 
where $\rho_n \ge 0$ for any $n \ge 1$.
It then follows that 
$\big(\gamma_n(u^{(k)})\big)_{n \ge 1}
\to (\rho_n)_{n \ge 1}$
 strongly in $\ell^1(\N, \R)$. Define
 $$
 \nu_n:= n - \sum_{p = n+1}^{\infty} \rho_p\,, \quad
 \forall n \ge 0\, .
 $$
Then for any $n \ge 1$, $\nu_n = \nu_{n-1} + 1 + \rho_n$
and $\lambda_n(u^{(k)}) \to \nu_n$ uniformly in $n \ge 0.$
Since $L_{u^{(k)}} \ge \lambda_0(u^{(k)})$
we infer that there exists $c > |-\nu_0 + 1|$
so that for any $k \ge 1$ and $\lambda \ge c$,
$$
L_{u^{(k)}} + \lambda : H^{1-s}_+ \to H^{-s}_+
$$
is a linear isomorphism whose inverse is bounded uniformly in $k$. Therefore
$$
w_\lambda^{(k)}:= 
\big(L_{u^{(k)}} + \lambda  \big)^{-1}[1]\, , \quad
\forall k \ge 1\, ,
$$
is a well-defined, bounded sequence in $H^{1-s}_+.$
Let us choose an arbitrary countable subset 
$\Lambda \subset [c, \infty)$ with one cluster point. By a diagonal procedure, we extract a subsequence
of $(w_\lambda^{(k)})_{k \ge 1}$, again denoted by
$(w_\lambda^{(k)})_{k \ge 1}$, so that for every 
$\lambda \in \Lambda,$ the sequence $(w_\lambda^{(k)})$
converges weakly in $H^{1-s}_+$ to some element 
$v_\lambda \in H^{1-s}_+$. By Rellich's theorem 
$$
\big(L_{u^{(k)}} + \lambda  \big) w_\lambda^{(k)}
\rightharpoonup 
\big(L_{u} + \lambda  \big) v_{\lambda}
$$
weakly in $H^{-s}_+$ as $k \to \infty.$
Since by definition, 
$\big(L_{u^{(k)}} + \lambda  \big) 
w_\lambda^{(k)} = 1$ for any $k \ge 1,$ it follows that
for any $\lambda \in \Lambda,$
$\big(L_{u} + \lambda  \big) 
v_\lambda = 1$ and thus by the definition of 
the generating function
$$
\mathcal H_\lambda (u^{(k)}) =
\langle w_\lambda^{(k)} | 1 \rangle \to
\langle v_\lambda  | 1 \rangle = 
\mathcal H_\lambda (u)\, , \quad
\forall \lambda \in \Lambda \, .
$$
Since $\mathcal H_\lambda (u^{(k)}) $ and 
$\mathcal H_\lambda (u)$ are meromorphic functions
whose poles are on the real axis, it follows that
the convergence holds for any 
$\lambda \in \C \setminus \R$. This proves
item $(i)$. \\
$(ii)$ We apply item $(i)$ (and its proof) as follows.
As mentioned above, 
$\lambda_n(u^{(k)}) \to \rho_n$, uniformly in $n \ge 0$.
By the proof of item $(i)$ one has for any $c \le \lambda < \infty,$
$$
\mathcal H_\lambda (u^{(k)})  \to 
\frac{1}{\nu_0+\lambda}\prod_{n=1}^\infty \big(1-\frac{\rho_n}{\nu_n+\lambda}   \big)\,  
$$
and we conclude that for any $\lambda \in \Lambda$
$$
\frac{1}{\lambda_0(u)+\lambda}\prod_{n=1}^\infty 
\big(1-\frac{\gamma_n(u)}{\lambda_n(u) +\lambda}   \big)
= \mathcal H_\lambda (u) =
\frac{1}{\nu_0+\lambda}\prod_{n=1}^\infty \big(1-\frac{\rho_n}{\nu_n+\lambda}   \big) .
$$
Since $\mathcal H_\lambda (u)$ and
infinite product are meromorphic functions
in $\lambda$, the functions are equal. In particular,
they have the same zeroes and the same poles.
Since the sequences 
$(\lambda_n(u))_{n \ge 0}$ and 
$(\nu_n(u))_{n \ge 0}$
are both listed in increasing order it follows that 
$\lambda_n(u) = \nu_n$ for any $n \ge 0$,
implying that for any $n \ge 1$, 
$$
\gamma_n(u) = 
\lambda_n(u) - \lambda_{n-1}(u) -1 = 
\nu_n - \nu_{n-1} -1 = \rho_n\, .
$$
By \eqref{lim gamma_n for -s} we then conclude that
$$
\big( \gamma_n(u^{(k)})^{1/2}\big)_{n \ge 1}
\rightharpoonup (\gamma_n(u)^{1/2})_{n \ge 1}
$$
weakly in $h^{1/2 -s}(\N, \R)$. 
\end{proof}

\begin{corollary}\label{kappa_n for neg s}
For any $0 \le s < 1/2$ and $n \ge 1$, the functional 
$\kappa_n : H^{-s}_{r,0} \to \R$,
introduced in \eqref{product for kappa for -s},
is sequentially weakly continuous.
\end{corollary}
\begin{proof}
Let $(u^{(k)})_{k \ge 1}$ be a sequence in 
$H^{-s}_{r,0}$ with
$u^{(k)}\rightharpoonup u$ weakly in $H^{-s}_{r,0}$ 
as $k \to \infty$. By \eqref{formula for lambda_n},
one has for any $p < n,$
$$
\lambda_p(u^{(k)}) - \lambda_n(u^{(k)}) =
p-n - \sum_{j = p+1}^n \gamma_j(u^{(k)})
$$
whereas for $p > n$
$$
\lambda_p(u^{(k)}) - \lambda_n(u^{(k)}) =
p-n + \sum_{j = n+1}^p \gamma_j(u^{(k)})\, .
$$
By Lemma \ref{gamma_n for -s}, one then concludes that
$$
\lim_{k \to \infty} \big( \lambda_p(u^{(k)}) - \lambda_n(u^{(k)})\big) - 
\big( \lambda_p(u) - \lambda_n(u)\big) =0
$$
uniformly in $p, n \ge 0$. By the product formula 
\eqref{product for kappa for -s}
for $\kappa_n$ it then follows that for any $n \ge 1$
$$
\lim_{k \to \infty} \kappa_n(u^{(k)}) = \kappa_n(u) \, .
$$
\end{proof} 

Furthermore, we need to prove the following lemma
concerning the eigenfunctions 
$f_n(\cdot, u)$, $n \ge 0$, of $L_u$.
\begin{lemma}\label{f_n bounded for s neg}
Given $0 \le s < 1/2$, $M > 0$ and $n \ge 0,$
there exists a constant $C_{s, M, n} \ge 1$ so that for any 
$u \in H^{-s}_{r,0}$ with $\|u \|_{-s} \le M$ and any $n \ge 0$
\begin{equation}\label{estimate f_n final}
\| f_n(\cdot , u)) \|_{1-s} \le C_{s, M, n} \, .
\end{equation}
\end{lemma}
\begin{proof}
By the normalisation of $f_n$,
 $\|f_n\| = 1$. Since $f_n$ is an eigenfunction, corresponding
 to the eigenvalue $\lambda_n,$ one has
 $$
  -i\partial_x f_n = L_u f_n + T_u f_n = 
 \lambda_n f_n + T_u f_n \, ,
 $$
implying that
\begin{equation}\label{estimate of f_n for neg s}
\| \partial_x f_n \|_{-s} \le |\lambda_n| + \|T_uf_n\|_{-s}
\, .
\end{equation}
Note that by the estimates \eqref{sandwich lambda_n}, 
\begin{equation}\label{estimate | lambda_n |}
|\lambda_n| \le \max\{ n , |\lambda_0| \}
\le n + |\lambda_0|\, , \quad
|\lambda_0 | \le 1+ \eta_s(\| u \|_{-s}) \, 
\end{equation}
where $\eta_s(\| u \|_{-s})$ is given by \eqref{def eta_s}.
Furthermore, since $\sigma = (1/2 +s)/2$ (cf. \eqref{para}) one has
$1-s>1- \sigma >1/2$, implying that $H_+^{1-\sigma}$ acts on $H_+^t$ 
for every $t$ in the open interval $(-(1-\sigma),1-\sigma)$. Hence 
\begin{equation}\label{estimate T_u f_n}
\| T_u f_n\|_{-s} \le C_{s}
\|u \|_{-s} \| f_n \|_{1-\sigma} \, .
\end{equation}
Using interpolation and Young's inequality
(cf. \eqref{interpolation}, \eqref{Young}),
\eqref{estimate T_u f_n} yields an estimate, which together
with \eqref{estimate of f_n for neg s} and 
\eqref{estimate | lambda_n |} leads to the claimed estimate 
\eqref{estimate f_n final}.
\end{proof}
 
With these preparations done, we can now prove
Proposition \ref{extension Phi, part 1}$(ii)$.

 \smallskip
 
 \noindent
{\em Proof of Proposition \ref{extension Phi, part 1}$(ii)$.}
 First we prove that for any $0 \le s < 1/2$,
$\Phi: H^{-s}_{r,0} \to h^{1/2 - s}_+$
is sequentially weakly continuous:
assume that $(u^{(k)})_{k \ge 1}$ is a sequence in 
$H^{-s}_{r,0}$ with
$u^{(k)} \rightharpoonup u$ weakly in $H^{-s}_{r,0}$ 
as $k \to \infty$. Let 
$\zeta^{(k)}:= \Phi(u^{(k)})$ and $\zeta := \Phi(u)$.
Since $(u^{(k)})_{k \ge 1}$ is bounded
in $H^{-s}_{r,0}$
and $\Phi$ maps bounded subsets of $H^{-s}_{r,0}$
to bounded subsets of $h^{1/2 -s}_+$,
the sequence $(\zeta^{(k)})_{k \ge 1}$ is bounded in 
$h^{1/2 - s}_+$. To show that 
$\zeta^{(k)}\rightharpoonup \zeta$ weakly in $h^{1/2}_+$,
it then suffices to prove that for any $n \ge 1,$
$\lim_{k \to \infty} \zeta^{(k)}_n = \zeta_n$.
By the definition of the Birkhoff coordinates 
\eqref{def coo for -s},
$\zeta^{(k)}_n =
\langle 1 | f_n^{(k)} \rangle / (\kappa_n^{(k)})^{1/2}$ 
where
$\kappa_n^{(k)}:= \kappa_n(u^{(k)})$ and 
$f_n^{(k)} := f_n(\cdot, u^{(k)})$.
By Corollary \ref{kappa_n for neg s},
$\lim_{k \to \infty}\kappa_n^{(k)} = \kappa_n$
and by Lemma \ref{f_n bounded for s neg},
saying that for any $n \ge 0$,  $\|f_n\|_{1-s}$
is uniformly bounded on bounded subsets of $H^{-s}_{r,0}$,
$\lim_{k \to \infty}\langle 1 | f_n^{(k)} \rangle =
\langle 1 | f_n \rangle$
where $\kappa_n := \kappa_n(u)$ and
$f_n := f_n(\cdot, u)$. This implies that 
$\lim_{k \to \infty} \zeta^{(k)}_n = \zeta_n$
for any $n \ge 1$.\\
It remains to show that for any $0 < s < 1/2,$
$\Phi : H^{-s}_{r,0} \to h^{1/2-s}_+$
is one-to-one. In the case where
$u \in L^2_{r,0}$,  it was verified
in the proof of \cite[Proposition 4.2]{GK}
that the Fourier coefficients $\widehat u(k)$, $k\ge 1$,
of $u$ can be explicitly expressed in terms of the components 
$\zeta_n(u)$ of the sequence $\zeta(u) = \Phi(u)$.
These formulas continue to hold for $u \in H^{-s}_{r,0}$.
This completes the proof of 
Proposition \ref{extension Phi, part 1}$(ii)$.
\hfill $\square$

\medskip

\section[extension of Birkhoff map. Part 2] {Extension of $\Phi$. Part 2}\label{extension of Phi, part 2}

In this section we prove the second part of 
Theorem \ref{extension Phi}, which we again state as a separate proposition.

\begin{proposition}\label{extension Phi, part 2}
{\sc (Extension of $\Phi$. Part 2)}
For any $0 < s < 1/2,$ the map 
$\Phi: H^{-s}_{r,0} \to h^{1/2-s}_+$  
has the following additional properties:\\
$(i)$ The inverse image of $\Phi$ of any
bounded subset of $h^{1/2-s}_+$ is a bounded
subset in 
$H^{-s}_{r,0}$. \\
$(ii)$ $\Phi$ is onto and the inverse map
$\Phi^{-1}: h^{1/2-s}_+ \to H^{-s}_{r,0}$,
is sequentially weakly continuous.\\
$(iii)$
For any $0 < s < 1/2,$ the Birkhoff map 
$\Phi: H^{-s}_{r,0} \to h^{1/2-s}_+$  
and its inverse 
$\Phi^{-1}: h^{1/2-s}_+ \to H^{-s}_{r,0}$,
are continuous. 
\end{proposition}
\begin{remark}
As mentioned in Remark \ref{weakPhi}, the map
$\Phi: L^2_{r,0} \to h^{1/2}_+$ and its inverse
$\Phi^{-1}: h^{1/2}_+ \to L^2_{r,0}$
are sequentially weakly continuous.
\end{remark}

\noindent
{\em Proof of Proposition \ref{extension Phi, part 2}$(i)$. }
Let $0 < s < 1/2$ and $u \in H^{-s}_{r,0}$.
Recall that by Corollary \ref{definition L_u},
$L_u$ is a self-adjoint operator with domain 
${\rm{dom}}(L_u) \subset H_+$,
has discrete spectrum and is bounded from below. 
Thus $L_u - \lambda_0(u) + 1 \ge 1$ where $\lambda_0(u)$
denotes the smallest eigenvalue of $L_u.$ By the considerations
 in Section \ref{extension of Phi, part 1}
 (cf. Lemma \ref{restriction of B_u to H^{1-s}}),
$L_u$ extends to a bounded operator 
$L_u: H^{1/2}_+ \to H^{-1/2}_+$ and satisfies 
$$
\langle L_u f | f \rangle = 
\langle D f | f \rangle - 
\langle u | f \overline f  \rangle\, , 
\quad \forall f \in H^{1/2}_+ \, .
$$
By Lemma \ref{estimate for T_u f}$(i)$ one has
$ | \langle  u | f \overline f \rangle |
 \le  C_{1,s}^2 \|u\|_{-s} \| f\|_{1/2}^2$ for any $f \in H^{1/2}_+$ and hence
 \begin{align}
 \| f \|^2 & \le 
 \langle ( L_u - \lambda_0(u) + 1) f | f \rangle
 \nonumber \\
 &\le \langle D f | f \rangle + 
 C_{1,s}^2 \|u\|_{-s}\| f\|_{1/2}^2
 + (- \lambda_0(u) + 1) \|f\|^2 \, , \nonumber
 \end{align}
yielding the estimate
$$
 \| f \|^2 \le  
 \langle ( L_u - \lambda_0(u) + 1) f | f \rangle \le
 M_u \|f\|_{1/2}^2
 $$
 where
\begin{equation}\label{def M_u}
 M_u:= C_{1,s}^2 \|u\|_{-s} + (2  - \lambda_0(u)) \,.
\end{equation}
To shorten notation, we will for the remainder of
the proof no longer indicate the dependence of
spectral quantities such as $\lambda_n$ or $\gamma_n$
on $u$ whenever appropriate.
The square root of the operator 
$L_u - \lambda_0 + 1$,
$$
R_u := (L_u - \lambda_0 + 1)^{1/2} : 
H^{1/2}_+ \to H_+\, ,
$$
can then be defined 
in terms of the basis $f_n \equiv f_n(\cdot, u)$,
$n \ge 0$, of eigenfunctions of $L_u$ in a standard way 
as follows: 
By Lemma \ref{basis on Sobolev scale}, any
$f \in H^{1/2}_+$ has an expansion of the form
$f = \sum_{n=0}^\infty \langle f | f_n \rangle f_n$
where $(\langle f | f_n \rangle)_{n \ge 0}$
is a sequence in $h^{1/2}(\N_0)$. $R_u f$ is then defined
as 
$$
R_u f := \sum_{n=0}^\infty 
(\lambda_n - \lambda_0 + 1)^{1/2}
\langle f | f_n \rangle f_n
$$
Since  $(\lambda_n - \lambda_0 + 1)^{1/2} 
\sim \sqrt{n}$ (cf. \eqref{sandwich lambda_n})
one has
$$
\big( (\lambda_n - \lambda_0 + 1)^{1/2}
\langle f | f_n \rangle)^{1/2}\big)_{n \ge 0}
\in \ell^2(\N_0)
$$
implying that $R_u f \in H_+$
(cf. Lemma \ref{basis on Sobolev scale}). 
Note that
$$
\|f\|^2 \le \langle R_u f | R_u f  \rangle
= \langle R_u^2 f | f  \rangle
\le M_u \|f\|^2_{1/2}\, , \quad
\forall f \in H^{1/2}_+\, ,
$$
and that $R_u$ is a positive self-adjoint operator
when viewed as an operator with domain $H^{1/2}_+$, 
acting on $H_+$.
By complex interpolation 
(cf. e.g. \cite[Section 1.4]{Tay}) one then concludes that
for any $0 \le \theta \le 1$
$$
R_u^{\theta} : H^{\theta/2}_+ \to H_+\, , \qquad
\| R_u^{\theta} f\|^2 \le 
M_u^\theta \|f \|^2_{\theta/2} \, , 
\quad \forall f \in H^{\theta/2}_+ \, .
$$
Since by duality, 
$$
R_u^{\theta} : H_+ \to H^{-\theta/2}_+ \, , \qquad
\| R_u^{\theta} g\|_{-\theta /2}^2 \le 
M_u^\theta \|g \|^2 \, , 
\quad \forall g \in H_+ \, ,
$$
one infers, 
using that $R_u^{\theta} : H_+ \to H^{-\theta/2}_+$
is boundedly invertible, that
for any $f \in H^{-\theta /2}_+$,
$$
\| f\|_{-\theta /2}^2 \le 
M_u^\theta \| R_u^{-\theta} f\|^2 \, ,
\qquad
R_u^{-\theta} := (R_u^{\theta})^{-1}\, .
$$
Applying the latter inequality
to $f =\Pi u$ and  $\theta = 2s$
and using that
$\Pi u = \sum_{n =1}^\infty 
\langle \Pi u | f_n\rangle f_n$
and   
$\langle \Pi u | f_n\rangle =
 - \lambda_n  \langle 1 | f_n\rangle$
one sees that
\begin{equation}\label{bound 1 for -s norm of u}
\frac{1}{2} \|u\|^{2}_{-s} = \|\Pi u\|^{2}_{-s}
\le M_u^{2s} \Sigma 
\end{equation}
where
$$
\Sigma:= \sum_{n=1}^\infty \lambda_n^{2}
\big( \lambda_n - \lambda_0+ 1 \big)^{-2s}
 \ |\langle 1\vert f_n\rangle |^2 \ .
 $$
We would like to deduce from 
\eqref{bound 1 for -s norm of u}
an estimate of $\|u\|_{-s}$ in terms of the
 $\gamma_n$'s.
Let us first consider $M_u^{2s}$. 
By \eqref{def M_u} one has
$$
M_u^{2s} = 2^{2s} \max \{ (C_{1,s}^2 \|u\|_{-s})^{2s},  
(2  - \lambda_0(u))^{2s}\} \, ,
$$ 
yielding
\begin{equation}\label{split M_u}
M_u^{2s}
\le (\|u\|_{-s}^2)^s (2 C_{1,s}^2)^{2s}
+ (2(2  - \lambda_0(u)))^{2s} \,.
\end{equation}
Applying Young's inequality 
with $1/p = s$, $1/q = 1 - s$
one obtains
\begin{equation}\label{bound first term of bound 1}
(\|u\|_{-s}^2)^s (2 C_{1,s}^2)^{2s} \Sigma
\le \frac{1}{4} \|u\|_{-s}^2 +
\big( (4 C_{1,s}^2)^{2s} \Sigma \big)^{1/(1-s)} \, ,
\end{equation}
which when combined with \eqref{bound 1 for -s norm of u} 
and \eqref{split M_u}, leads to
$$
\frac{1}{4} \|u\|_{-s}^2 \le 
\big( (4 C_{1,s}^2)^{2s} \Sigma \big)
^{1/(1-s)} 
+ (2(2  - \lambda_0(u)))^{2s} \Sigma \, .
$$
The latter estimate is of the form
\begin{equation}\label{bound 2 of s norm of u}
\|u\|_{-s}^2 \le 
C_{3,s} \Sigma^{1/(1-s)} 
+C_{4,s} (2  - \lambda_0(u))^{2s} \Sigma \, ,
\end{equation}
where $ C_{3,s}, C_{4,s} > 0$ are constants, only
depending on $s$.
Next let us turn to 
$\Sigma = \sum_{n=1}^\infty \lambda_n^{2}
\big( \lambda_n - \lambda_0+ 1 \big)^{-2s}
 \ |\langle 1\vert f_n\rangle |^2$.
Since 
$$
\lambda_n = n - \sum_{k=n+1}^\infty \gamma_k \ ,
\qquad 
\ |\langle 1\vert f_n\rangle |^2 = \gamma_n \kappa_n \, .
$$
and 
\begin{equation}\label{product for kappa}
\kappa_n = \frac{1}{\lambda_n - \lambda_0}
\prod_{p \ne n} (1 - \frac{\gamma_p}{\lambda_p - \lambda_n})\, ,
\end{equation}
the series $\Sigma$ can be expressed in terms of the $\gamma_n$'s. To obtain a bound for $\Sigma$
it remains to estimate the $\kappa_n$'s. Note that
$$
\prod_{p \ne n} 
(1 - \frac{\gamma_p}{\lambda_p - \lambda_n})
\le \prod_{p < n} 
(1 + \frac{\gamma_p}{\lambda_n - \lambda_p})
\le e^{\sum_{p=1}^n \gamma_p} \le e^{- \lambda_0} \, .
$$
Since $ (\lambda_n - \lambda_0)^{-1} = 
(n + \sum_{k=1}^n \gamma_k)^{-1} \le n^{-1}$,
it then follows that
$$
0 < \kappa_n \le  \ \frac{e^{- \lambda_0}}{n}\, , \quad \forall n \ge 1.
$$
Combining the estimates above we get
$$
\Sigma \le e^{-\lambda_0} 
\sum_{n=1}^\infty \lambda_n^{2} n^{-2s -1}  \gamma_n  \ .
$$
By splitting the sum $\Sigma$
into two parts, 
$\Sigma = 
\sum_{ n < - \lambda_0(u) } + 
\sum_{n \ge - \lambda_0(u) }
$ 
and taking into account that $0 \le \lambda_n \le n$ for any 
$n \ge - \lambda_0$ and  $|\lambda_n| \le -\lambda_0 $
for any $1 \le n < - \lambda_0$, one has
$$
\Sigma \le (1 - \lambda_0)^2e^{-\lambda_0} 
\sum_{n = 1}^\infty n^{1 -2s} \gamma_n \, .
$$
Together with the estimate 
\eqref{bound 2 of s norm of u} this shows that
the inverse image  by $\Phi$  
of any bounded subset 
of sequences in $h^{1/2 -s}$ is bounded 
in $H^{-s}_{r,0}$.
\hfill $\square$

\medskip

\noindent
{\em Proof of Proposition \ref{extension Phi, part 2}$(ii)$. }
 First we prove that for any $0 < s < 1/2$,
$\Phi: H^{-s}_{r,0} \to h^{1/2-s}_+$ is onto.
Given $z=(z_n)_{n \ge 1}$ in $h^{1/2-s}_+,$ consider
the sequence $\zeta^{(k)} = (\zeta^{(k)}_n)_{n \ge 1}$,
defined for any $k \ge 1$ by
$$
\zeta^{(k)}_n = z_n\, \,\,\, \forall 1 \le n \le k\, , \qquad
\zeta^{(k)}_n = 0\, \,\,\, \forall n > k\, .
$$
Clearly $\zeta^{(k)} \to z$ strongly in $h^{1/2-s}$.
Since for any $k \ge 1$, $\zeta^{(k)} \in h^{1/2}_+$ 
Theorem \ref{main result} implies that there exists 
a unique element
$u^{(k)}\in L^2_{r,0}$ with $\Phi(u^{(k)})=\zeta^{(k)}$.
By Proposition \ref{extension Phi, part 2}$(i)$, 
$\sup_{k \ge 1}\|u^{(k)} \|_{-s} < \infty$.
Choose a weakly convergent subsequence 
$(u^{(k_j)})_{j \ge 1}$ of $(u^{(k)})_{k \ge 1}$
and denote its weak limit in $H^{-s}_{r,0}$ by $u$.
Since by Proposition \ref{extension Phi, part 1},
$\Phi: H^{-s}_{r,0} \to h^{1/2-s}_+$ is 
sequentially weakly continuous,
$\Phi(u^{(k_j)}) \rightharpoonup \Phi(u)$
weakly in $h^{1/2-s}_+$. On the other hand,
$\Phi(u^{(k_j)}) = \zeta^{(k_j)} \to z$ strongly in $h^{1/2-s}$,
implying that $\Phi(u) = z$.
 This shows that $\Phi$ is onto.\\
 It remains to prove that for any $0 \le s < 1/2$, $\Phi^{-1}$
is sequentially weakly continuous. 
Assume that $(\zeta^{(k)})_{k \ge 1}$ is a sequence in $h^{1/2-s}$,
weakly converging to $\zeta \in h^{1/2-s}$.
Let $u^{(k)}:= \Phi^{-1}(\zeta^{(k)})$.
By Proposition \ref{extension Phi, part 2}$(i)$
(in the case $0 < s < 1/2$) and
Remark \ref{RemarkThm1}$(ii)$ (in the case $s=0$),
 $(u^{(k)})_{k \ge 1}$ is a bounded sequence in
$H^{-s}_{r,0}$ and thus admits a weakly convergent
subsequence $(u^{k_j})_{j \ge 1}$. Denote its
limit in $H^{-s}_{r,0}$ by $u$.
Since by Proposition \ref{extension Phi, part 1}, $\Phi$
is sequentially weakly continuous, $\Phi(u^{k_j}) 
\rightharpoonup \Phi(u)$
weakly in $h^{1/2-s}$. 
On the other hand, by assumption,
$\Phi(u^{k_j}) = \zeta^{(k_j)} \rightharpoonup \zeta$
and hence $u = \Phi^{-1}(\zeta)$
and $u$ is independent
of the chosen subsequence $(u^{k_j})_{j \ge 1}$.
This shows that 
$\Phi^{-1}(\zeta^{(k)}) \rightharpoonup  \Phi^{-1}(\zeta)$
weakly in $H^{-s}_{r,0}$.
\hfill $\square$

\medskip

\noindent
{\em Proof of 
Proposition \ref{extension Phi, part 2}$(iii)$. }
By Proposition \ref{extension Phi, part 1},
$\Phi:  H^{-s}_{r,0} \to h^{1/2-s}_+$ is 
sequentially weakly continuous for any $0 \le s < 1/2$. 
To show that this map is continuous
it then suffices to prove that
the image $\Phi(A)$ of any
relatively compact subset $A$ of $ H^{-s}_{r,0}$ 
 is relatively compact in $h^{1/2-s}_+$.
For any given $\e > 0,$ choose $N \equiv N_\e \ge 1$ and
$R \equiv R_\e > 0$ as in 
Lemma \ref{rel compact subsets}, stated below.
Decompose $u \in A$ as $u = u_N + u_{\bot}$ where
$$
u_N := \sum_{0 <|n| \le N_\e} \widehat u(n) e^{inx} \, ,
\qquad
u_\bot := \sum_{|n| > N_\e} \widehat u(n) e^{inx} \, .
$$
By Lemma \ref{rel compact subsets}, $\|u_N\| < R_\e$
and $\| u_\bot \|_{-s} < \e.$ 
By Lemma \ref{basis on Sobolev scale}, applied
with $\theta = -s$, one has
$$
K_{u; -s}(\Pi u) = K_{u; -s}(\Pi u_N) +
K_{u; -s}(\Pi u_\bot) \in h^{-s}(\N_0)
$$
where  
$K_{u; -s}(\Pi u_N) = K_{u; 0}(\Pi u_N)$
since $\Pi u_N \in H_+$. 
Lemma \ref{basis on Sobolev scale} then implies 
that there exists $C_A > 0$, independent of $u \in A$, 
so that
$$
\| K_{u; 0}(\Pi u_N) \| \le C_A R_\e \, , \qquad
\| K_{u; -s}(\Pi u_\bot) \|_{-s} \le C_A \e \, .
$$
Since $\e > 0$ can be chosen arbitrarily small,
it then follows by Lemma \ref{rel compact subsets}
that $K_{u; -s}(\Pi (A))$ is relatively compact in
$h^{-s}(\N_0)$. Since by definition
$$
\big( K_{u; -s}(\Pi u) \big)_n = 
\langle \Pi u | f_n (\cdot, u) \rangle
\, , \quad \forall n \ge 0\, ,
$$ 
and since by \eqref{formula coeff of u},
$$
\zeta_n(u) \simeq \frac{1}{\sqrt{n}}
\langle \Pi u | f_n (\cdot, u) \rangle   \quad \text{as }\,\,  n \to \infty
$$
uniformly with respect to $u \in A,$ it follows that $\Phi(A)$
is relatively compact in $h^{1/2 -s}_+$.\\
Now let us turn to $\Phi^{-1}$. 
By Proposition \ref{extension Phi, part 2}(ii),
$\Phi^{-1}: h^{1/2-s}_+ \to H^{-s}_{r,0}$
is sequentially weakly continuous. 
To show that this map is continuous
it then suffices to prove that the image $\Phi^{-1}(B)$ of any
relatively compact subset $B$ of $h^{1/2-s}_+$
 is relatively compact in $ H^{-s}_{r,0}$.
By the same arguments as above one sees that 
$\Phi^{-1}: h^{1/2-s}_+ \to H^{-s}_{r,0}$ 
is also continuous.
\hfill $\square$

\smallskip

It remains to state Lemma \ref{rel compact subsets},
used in the proof of  
Proposition \ref{extension Phi, part 2}$(iii)$.
It concerns the well known characterization
of relatively compact subsets of $H^{-s}_{r,0}$
in terms of the Fourier expansion
$u(x) = \sum_{n \ne 0} \widehat u(n) e^{inx}$
of an element $u$ in $H^{-s}_{r,0}$.
\begin{lemma}\label{rel compact subsets}
Let $0 < s < 1/2$ and $A \subset H^{-s}_{r,0}$.
Then $A$ is relatively compact in $H^{-s}_{r,0}$
if and only if for any $\e > 0,$ there exist
$N_\e \ge 1$ and $R_\e > 0$  so that for any $u \in A,$
$$
\big(\sum_{|n| > N_\e} |n|^{-2s} |\widehat u(n)|^2\big)^{1/2} 
< \e
\, , \qquad 
\big( \sum_{0 <|n| \le N_\e} |\widehat u(n)|^2 \big)^{1/2} 
< R_\e \, .
$$
The latter conditions characterize relatively compact subsets of $h^{-s}(\N_0)$.
\end{lemma}

\smallskip
\noindent
{\em Proof of Theorem \ref{extension Phi}.}
The claimed statements follow from
Proposition \ref{extension Phi, part 1} and
Proposition \ref{extension Phi, part 2}.
In particular, item $(ii)$ of Theorem \ref{extension Phi} 
follows from Proposition \ref{extension Phi, part 2}$(i)$.
\hfill $\square$
 

\section{Solution maps $\mathcal S_0$, $\mathcal S_B$ and $\mathcal S_c$, $\mathcal S_{c,B}$} \label{mathcal S_B}

In this section we provide results related to the solution
map of \eqref{BO}, which will be used to prove 
Theorem \ref{Theorem 1} in the subsequent section.

\subsection{Solution map $\mathcal S_{B}$ and its extension}
First we study the map $\mathcal S_B,$ defined in Section 2
on $h^{1/2}_+$. Recall that by 
\eqref{frequencies in Birkhoff} -- \eqref{frequency map},
the nth frequency of \eqref{BO} is a real valued map defined on $\ell^2_+$ by
$$
\omega_n(\zeta) := n^2 -2\sum_{k=1}^{n} k |\zeta_k|^2
- 2n\sum_{k=n+1}^{\infty} |\zeta_k|^2\, .
$$
For any $0 < s \le 1/2$, the map $\mathcal S_B$ naturally extends to $h^{1/2 -s}_+$, mapping initial data 
$\zeta(0) \in h^{1/2 -s}_+$ to the curve 
\begin{equation}\label{formula for S_B}
\mathcal S_B(\cdot, \zeta(0)): \R \to h^{1/2 -s}_+\ , \  t \mapsto 
\mathcal S_B(t, \zeta(0)):=
\big( \zeta_n(0)e^{it\omega_n(\zeta)} \big)_{n \ge 1} \, .
\end{equation}
We first record the following properties of the frequencies.
\begin{lemma}\label{continuity frequency map} $(i)$ For any $n \ge 1$,
 $\omega_n: \ell^2_+ \to \R$ is continuous and
$$
| \omega_n(\zeta) - n^2 | \le 2 n \|\zeta\|_0^2 \, , 
 \,\,\, \forall \zeta \in  \ell^2_+\, ;
 \quad 
 | \omega_n(\zeta) - n^2 | \le  2\|\zeta\|_{1/2}^2 
\, , \,\,\, \forall \zeta \in  h^{1/2}_+\, .
$$
$(ii)$ For any $0 \le s < 1/2$, \, 
 $\omega_n: h^{1/2 - s}_+ \to \R$
is sequentially weakly continuous.
\end{lemma}
\begin{proof}
Item $(i)$ follows in a straightforward way from the formula
\eqref{frequencies in Birkhoff} of $\omega_n$. 
Since for any $0 \le s < 1/2$, $h^{1/2 - s}_+$ 
compactly embeds into $\ell^2_+$, item $(ii)$ follows from $(i)$.
\end{proof}
From Lemma \ref{continuity frequency map} one infers  the following properties of $\mathcal S_B$. We leave the easy proof to the reader.
\begin{proposition}\label{S in Birkhoff continuous}
For any $0 \le s \le 1/2,$ the following holds:\\
$(i)$ For any  initial data $\zeta(0) \in h^{1/2-s}_+$, the curve
$$
\R \to h^{1/2-s}_+\, , \, t \mapsto 
\mathcal S_B(t, \zeta(0))
$$ 
is continuous.\\
$(ii)$ For any $T > 0,$
$$
\mathcal S_B: h^{1/2-s}_+ \to C([-T, T], h^{1/2-s}_+), \, 
\zeta(0) \mapsto  \mathcal S_B(\cdot, \zeta(0))\, ,
$$
is continuous. For any $t \in \R$,
$$
\mathcal S^t_B : h^{1/2-s}_+ \to h^{1/2-s}_+, \,
\zeta(0) \mapsto  \mathcal S_B(t, \zeta(0)) \, ,
$$
is a homeomorphism.
\end{proposition} 

\subsection{Solution map $\mathcal S_0$ and its extension}
Recall that in Section \ref{Birkhoff map} we introduced the
solution map $\mathcal S_0$ of \eqref{BO}
on the subspace space $L^2_{r,0}$ of $L^2_r$,
consisting of elements in $L^2_r$ with average $0$,
in terms of the Birkhoff map $\Phi$,
\begin{equation}\label{flow map BO}
\mathcal S_0 = \Phi^{-1} \mathcal S_B \Phi\, : \,  
L^2_{r,0} \to C(\R, L^2_{r,0}) \, .
\end{equation}
Theorem \ref{extension Phi} will now be applied to
prove the following result about the extension
of $\mathcal S_0$ to the Sobolev space $H^{-s}_{r,0}$
with $0 < s < 1/2$, consisting of elements in
$H^{-s}_r$ with average zero.
It will be used in Section \ref{Proofs of main results} 
to prove Theorem \ref{Theorem 1}.

\begin{proposition}\label{solution map S_0 for neg}
For any $0 \le s < 1/2$, the following holds:\\
$(i)$ The Benjamin-Ono equation is globally $C^0-$well-posed on $H^{-s}_{r,0}$.\\
$(ii)$ There exists an increasing function
$F_s :\R_{\ge 0} \to \R_{\ge 0} $ so that
$$
\| u \|_{-s} \le F_s\big(\|\Phi(u)\|_{1/2-s}\big)\, , \quad \forall u \in H^{-s}_{r,0} \, .
$$
In particular, 
for any initial data $u(0) \in H^{-s}_{r,0}$,
\begin{equation}\label{a priori bound}
\sup_{t\in \R}\| \mathcal S^t_0(u(0)) \|_{-s}\leq 
F_s\big(\|\Phi(u(0))\|_{1/2-s}\big)\, .
\end{equation}
\end{proposition}
\begin{remark}
$(i)$ By the trace formula \eqref{trace formula},
for any $u(0) \in L^2_{r,0}$, 
estimate \eqref{a priori bound} can be improved as follows,
$$
\|u(t) \| = \sqrt{2} \|\Phi(u(0))\|_{1/2} =
\|u(0)\|\, , \quad \forall t \in \R\, .
$$
$(ii)$ We refer to the comments of Theorem \ref{Theorem 1}
in Section \ref{Introduction} for a discussion of the
recent results of Talbut \cite{Tal}, related to \eqref{a priori bound}.
\end{remark}
\begin{proof} Statement $(i)$  follows from the corresponding statements for $S_B$ in Proposition \ref{S in Birkhoff continuous}
and the continuity properties of $\Phi $ and $\Phi^{-1}$ stated in Theorem \ref{extension Phi}. \\
$(ii)$ By Theorem \ref{extension Phi} there exists an increasing
function $F_s : \R_{\ge 0} \to \R_{\ge 0}$ so that for any
$u \in H^{-s}_{r,0}$, 
$\| u \|_{-s} \le F_s(\| \Phi(u) \|_{1/2 -s})$.
Since the norm of $h^{1/2 -s}$
is left invariant by the flow $\mathcal S_B^t$ , 
it follows that 
for any initial data $u(0) \in H^{-s}_{r,0}$, one has
$
\sup_{t\in \R}\| \mathcal S^t(u(0))\|_{-s}\leq 
F_s\big(\|\Phi(u(0))\|_{1/2-s}\big)\, .
$
\end{proof}


\subsection{Solution map $\mathcal S_{c}$}\label{solution map Sc}
Next we introduce the solution map $\mathcal S_c$
where $c$ is a real parameter.
Let $v(t,x)$ be a solution of \eqref{BO} with initial data $v(0) \in H^s_r$ and $s > 3/2$,
satisfying the properties $(S1)$ and $(S2)$ stated in
Section \ref{Introduction}. 
By the uniqueness property in $(S1)$, it then follows that
 \begin{equation}\label{identity of solutions}
 v(t, x) = u(t, x - 2ct) + c\, , \quad  c= \langle v(0) | 1 \rangle
 \end{equation}
 where $u \in C(\R, H^s_{r,0}) \cap C^1(\R, H^{s-2}_{r,0})$
 is the solution of the initial value problem
 \begin{equation}\label{BO with mean zero}
\partial_t u = H\partial^2_x u - \partial_x (u^2)\,, \qquad  u(0) = v(0) - \langle v(0) | 1 \rangle \, ,
\end{equation}
satisfying $(S1)$ and $(S2)$. 
It then follows that
$ w(t,x) := u(t, x -2ct)$ satisfies $w(0) = u(0)$ and
\begin{equation}\label{BO with c}
\partial_t w = H\partial^2_x w - \partial_x (w^2)  + 2c \partial_x w \, .
\end{equation}
By \eqref{identity of solutions},
the solution map of \eqref{BO with c}, denoted by $\mathcal S_c$, 
is related to the solution map $\mathcal S$ of \eqref{BO}
(cf. property $(S2)$ stated in Section \ref{Introduction}) by
\begin{equation}\label{formula with c}
\mathcal S(t, v(0)) = 
\mathcal S_{[v(0)]} \big(t, v(0) - [v(0)] \big) + [v(0)]
\ , \quad [v(0)] := \langle v(0) | 1 \rangle \, .
\end{equation}
In particular, for any $s > 3/2$,
\begin{equation}\label{solution map for BO with c}
\mathcal S_c: H^s_{r,0} \to 
C(\R, H^s_{r, 0}), 
w(0) \mapsto \mathcal S_c(\cdot, w(0) ) 
\end{equation}
is well defined and continuous.
Molinet's results in \cite{Mol} ( cf. also \cite{MP}) imply that the solution map $\mathcal S_c$ continuously 
extends to any
Sobolev space $H^s_{r,0}$ with $0 \le s \le 3/2$. More precisely, for any such $s$,
$\mathcal S_c: H^s_{r,0} \to 
C(\R, H^s_{r,0})$ is continuous
and for any $v_0 \in H^s_{r,0}$, $ \mathcal S_c(t, w_0)$ satisfies equation \eqref{BO}  
in $H^{s-2}_r$.

\subsection{Solution map $\mathcal S_{c, B}$ and its extension}
Arguing as in Section \ref{Birkhoff map}, we use
Theorem \ref{main result}, to express
the solution map $\mathcal S_{c, B}$,
corresponding to the equation \eqref{BO with c}
in Birkhoff coordinates.
Note that \eqref{BO with c} is Hamiltonian, $\partial_t w = \partial_x \nabla \mathcal H_c$, with Hamiltonian
$$
\mathcal H_c : H^s_{r,0} \to \R \, , \qquad  \mathcal H_c (w) =  \mathcal H(w)  + 2c \mathcal H^{(0)}(w)
$$
where by \eqref{CL},  $\mathcal H^{(0)}(w)  = \frac{1}{2\pi} \int_0^{2\pi} \frac{1}{2} w^2 dx$.
Since by Parseval's formula, derived  in \cite[Proposition 3.1]{GK}, 
$$
 \frac{1}{2\pi} \int_0^{2\pi} \frac{1}{2} w^2 dx = \sum_{n=1}^\infty n | \zeta_n |^2
 $$
one has
$$
\mathcal H_{c, B} (\zeta) := \mathcal H_c (\Phi^{-1} (\zeta)) = \mathcal H_{B} (\zeta) +2c  \sum_{n=1}^\infty n | \zeta_n |^2 \, ,
$$
implying that the corresponding frequencies $\omega_{c, n} $, $n \ge 1$, are given by
$$
\omega_{c, n} (\zeta) =   \partial_{|\zeta_n|^2} \mathcal H_{c, B} (\zeta)= \omega_{n} (\zeta) + 2c n \, .
$$
For any $c \in \R$, denote by $\mathcal S_{c, B}$
the solution map of \eqref{BO with c}
when expressed in Birkhoff coordinates,
$$
\mathcal S_{c, B} : h^{1/2}_+ \to C(\R,  h^{1/2}_+ )\, , \zeta(0)  \mapsto  [ t \mapsto  \big(  \zeta_n(0) e^{i t \omega_{c,n}(\zeta(0))}\big)_{n \ge 1} ] \,.
$$
Note that $\omega_{0,n} = \omega_n$ and hence
$\mathcal S_{0, B}  =  \mathcal S_{ B}$. 
Using the same arguments as in the proof of 
Proposition \ref{S in Birkhoff continuous} one obtains the following
\begin{corollary}\label{solution map S_{c,B} for neg}
The statements of Proposition \ref{S in Birkhoff continuous}
continue to hold for $\mathcal S_{c, B}$ with $c \in \R$
arbitrary.
\end{corollary}

\subsection{Extension of the solution map $\mathcal S_{c}$}
Above, we introduced the
solution map $\mathcal S_c$
on the subspace space $L^2_{r,0}$. One infers from
\eqref{formula with c} that
\begin{equation}\label{flow map BO}
\mathcal S_c = \Phi^{-1} \mathcal S_{c,B} \Phi\, : \,  
L^2_{r,0} \to C(\R, L^2_{r,0}) \, .
\end{equation}
Using the same arguments as in the proof of 
Proposition \ref{solution map S_0 for neg} 
one infers from Corollary \ref{solution map S_{c,B} for neg} 
the following results, concerning the extension
of $\mathcal S_c$ to the Sobolev space $H^{-s}_{r,0}$
with $0 < s < 1/2$.
\begin{corollary}\label{solution map S_c for neg}
The statements of Proposition \ref{solution map S_0 for neg}
continue to hold for $\mathcal S_{c, B}$ with $c \in \R$
arbitrary.
\end{corollary}


\section[Proofs]{Proofs of the main results} \label{Proofs of main results}

\noindent
{\em Proof of Theorem \ref{Theorem 1}.}
Theorem \ref{Theorem 1} is a straightforward consequence
of Proposition \ref{solution map S_0 for neg} and 
Corollary \ref{solution map S_c for neg}.
\hfill $\square$

\medskip
\noindent
{\em Proof  of Remark \ref{illposedness} (Illposedness of \eqref{BO} in $H^{-s}$ for  $s>\frac 12$).}
Since the general case can be proved by the same arguments
we consider only the case $c = 0$. Let $t \ne 0$. To show
that the solution map $S^t$ cannot be extended to 
$H^{-s}_{r,0}$ for $s > 1/2$, we will study 
one gap solutions.
Without further reference, we use notations and results from \cite[Appendix B]{GK}, where one gap potentials have
been analyzed. Consider the following family
of one gap potentials of average zero, parametrized by
$0 \le q < 1$,
$$
u_{0,q}(x)=2{\rm{Re}} \big( q e^{ix}/(1-qe^{ix})\big)\, , \quad
0 < q < 1\, .
$$
The gaps $\gamma_n(u_{0,q})$, $n \ge 1$, of $u_{0,q}$ can be computed as
$$
\gamma_{1, q} := \gamma_1(u_{0,q}) = q^2/(1-q^2) \, , \qquad
\gamma_n(u_{0,q}) = 0 \, , \quad \forall n \ge 2 \, . 
$$
The 
frequency $\omega_{1, q} := \omega_1 (u_{0,q})$ 
is thus given by (cf. \eqref{frequencies in Birkhoff})
$$
\omega_{1,q} = 1 -2\gamma_{1,q} = \frac{1-3q^2}{1 - q^2}.
$$
The one gap solution, also referred to as travelling wave solution, of the BO equation
with initial data $u_{0,q}$ is then given by
$$
u_q(t,x)=u_{0,q}(x+\omega_{1,q}t)\, , 
\quad \forall t \in \R\, .
$$
Note that for any $s > 1/2$, 
$$
\lim_{q \to 1}u_{0,q} = 
2 {\rm{Re}} (\sum_{k=1}^{\infty} e^{ikx})
= \delta_0 - 1
$$ 
strongly in $H^{-s}_{r,0}$
where $\delta_0$ denotes the periodic Dirac $\delta-$distribution, centered at $0$. 
Since $\omega_{1,q} \to - \infty$ as $q \to 1$,
it follows that for any $t \ne 0,$
$u_q(t, \cdot)$ does {\em not} converge
in the sense of distributions as $q \to 1$.
\hfill $\square$

\medskip
\noindent
{\em Proof of Theorem \ref{Theorem 3}.}
We argue similarly as in the proof of \cite[Theorem 2]{GK}.
Since the case $c \ne 0$ is proved be the same 
arguments we only consider the case $c=0$.
Let $u_0 \in H^{-s}_{r,0}$ with $0\le s < 1/2$
and let $u(t):= \mathcal S_0(t, u_0)$. By formula  
\eqref{formula for S_B}, 
$\zeta(t) := \mathcal S_B(t, \Phi(u_0))$ evolves
on the torus ${\rm{Tor}}(\Phi(u_0))$, defined by \eqref{def tor}.\\
$(i)$ Since ${\rm{Tor}}(\Phi(u_0))$ 
is compact in $h^{1/2 -s}_+$
and $\Phi^{-1}: h^{1/2 -s}_+ \to H^{-s}_{r,0}$
is continuous, $\{ u(t) \ : \ t \in \R \}$
is relatively compact in $H^{-s}_{r,0}$.\\
$(ii)$
In order to prove that $ t \mapsto u(t)$ is almost periodic, we appeal to Bochner's  characterization of such functions (cf. e.g. \cite{LZ}) : a bounded continuous function $f:\R \to X$
with values in a Banach space $X$ is almost periodic if and only if the set $\{ f_\tau, \tau \in \R\}$  of functions defined by 
$f_\tau (t):=f(t+\tau)$
is relatively compact in the space $\mathcal C_b(\R, X)$ of bounded continuous functions on $\R$ with values in $X$.
Since $\Phi: H^{-s}_{r,0} \to h^{1/2 -s}_+ $ 
is a homeomorphism, in the case at hand, it suffices to prove that for every sequence $(\tau_k)_{k\ge 1}$ of real numbers,  the sequence 
$f_{\tau_k}(t):=\Phi (u(t + \tau_k))$, $k\ge 1,$ 
in $C_b(\R,h^{1/2 -s}_+)$ admits a subsequence which converges uniformly in
$C_b(\R,h^{1/2 - s}_+)$. Notice that
$$f_{\tau_k}(t)=\big(\zeta_n(u(0)){\rm e}^{i\omega_n(t+\tau_k)}\big)_{n\ge 1}.$$
 By Cantor's diagonal process
and since the circle is compact, 
there exists a subsequence of $(\tau_k)_{k\ge 1}$,
again denoted by $(\tau_k)_{k\ge 1}$, so that for any $n \ge 1,$ $\lim_{k \to \infty}{\rm e}^{i\omega_n\tau_k}$ exists,
implying that the sequence of functions $f_{\tau_k}$
converges 
uniformly in $C_b(\R,h^{1/2 -s}_+)$.
\hfill $\square$

\medskip

\noindent
{\em Proof of Theorem \ref{Theorem 4}.}
Since the general case can be proved by the same arguments
we consider only the case $c = 0$.
By  \cite[Proposition B.1]{GK}, the travelling wave solutions
of the BO equation on $\T$ coincide with the one gap solutions.
Without further reference, we use notations and results from \cite[Appendix B]{GK}, where one gap potentials have
been analyzed. 
Let $u_0$ be an arbitrary one gap potential. Then
$u_0$ is $C^{\infty}-$smooth and there
exists $N \ge 1$ so that
$\gamma_N(u_0) > 0$ and $\gamma_n(u_0) = 0$ for any
$n \ne N$. 
Furthermore, the orbit of the corresponding one gap solution is given by $\{ u_0(\cdot + \tau) \, : \, \tau \in \R \}$.
Let $0 \le s < 1/2$. 
It is to prove that for any $\e > 0$ there exists $\delta > 0$
so that for any $v(0) \in H^{-s}_{r,0}$ with 
$\|v(0) - u_0 \|_{-s} < \delta$ one has
\begin{equation}\label{epsilon orbit close}
\sup_{t \in \R} \inf_{\tau \in \R}
\| v(t) - u_0( \cdot + \tau) \|_{-s} < \e \, .
\end{equation}
To prove the latter statement, we argue by contradiction.
Assume that there exists $\e > 0$, a sequence 
$(v^{(k)}(0))_{k \ge 1}$ in $H^{-s}_{r,0}$, and
a sequence $(t_k)_{k \ge 1}$  so that
$$
\inf_{\tau \in \R} 
\| v^{(k)}(t_k) - u_0( \cdot + \tau) \|_{-s} \ge \e \, , 
\quad \forall k \ge 1 \, , \qquad
\lim_{k \to \infty } \| v^{(k)}(0) - u_0 \|_{-s} = 0 \, .
$$
Since $A:= \{v^{(k)}(0) \, | \, k \ge 1 \} \cup \{ u_0 \}$
is compact in $H^{-s}_{r,0}$ and $\Phi$ is continuous, 
$\Phi(A)$ is compact in $h^{1/2-s}_+$ and
$$
\lim_{k \to \infty} 
\| \Phi(v^{(k)}(0)) - \Phi(u_0) \|_{1/2 -s} = 0 \, .
$$
It means that
$$
\lim_{k \to \infty} \sum_{n = 1}^{\infty}
n^{1 - 2s} | \zeta_n(v^{(k)}(0)) - \zeta_n(u_0) |^2 = 0 \, .
$$
Note that for any $k \ge 1$,
$$
\zeta_n(v^{(k)}(t_k)) = 
\zeta_n(v^{(k)}(0)) \ e^{i t_k \omega_n(v^{(k)}(0))}\ ,
\quad \forall n \ge 1
$$
and $\zeta_n(u(t_k)) = \zeta_n(u_0) = 0$ 
for any $n \ne N$. Hence
\begin{equation}\label{estimate for n ne N}
\lim_{k \to \infty} \sum_{n \ne  N}
n^{1 - 2s} | \zeta_n(v^{(k)}(t_k)) |^2 = 0\ , 
\end{equation}
and since 
$| \zeta_N(v^{(k)}(t_k))| = | \zeta_N(v^{(k)}(0))|$ 
one has
\begin{equation}\label{estimate for n equal N}
\lim_{k \to \infty} \big| | \zeta_N(v^{(k)}(t_k))|
-  | \zeta_N(u_0) | \big| = 0\ ,
\end{equation}
implying that 
$\sup_{k \ge 1} | \zeta_N(v^{(k)}(t_k))| < \infty$.
It thus follows that the subset $\{ \Phi(v^{(k)}(t_k))  : k \ge 1 \}$
is relatively compact in $h^{1/2 -s}$ and hence 
$\{ v^{(k)}(t_k) :  k \ge 1 \}$ relatively compact 
in $H^{-s}_{r,0}$. Choose a subsequence 
$( v^{(k_{j})}(t_{k_{j}}))_{j \ge 1}$ which
converges in $H^{-s}_{r,0}$ and denote its limit
by $w \in H^{-s}_{r,0}$. By 
\eqref{estimate for n ne N}--\eqref{estimate for n equal N}
one infers that there exists $\theta \in \R$ so that 
$$
\zeta_n(w) =0\ , \quad \forall n \ne N\, , \qquad
\zeta_N(w) = \zeta_{N}(u_0)e^{i \theta}\ .
$$
As a consequence, $w(x) = u_0(x+ \theta / N)$, contradicting
the assumption that 
$\inf_{\tau \in \R} 
\| v^{(k)}(t_k) - u_0( \cdot + \tau) \|_{-s} \ge \e$
for any $k \ge 1$.
\hfill $\square$


\end{document}